\documentclass[12pt]{amsart}
\date{4 February 2021}
\usepackage{latexsym,amsmath,amsthm,amsfonts,amscd,amssymb,mathdots}
\usepackage{stmaryrd}
\usepackage[mathscr]{eucal} 
\usepackage{mathrsfs}
\usepackage[utf8]{inputenc}
\usepackage[all,cmtip]{xy}
\usepackage{hyperref}
\usepackage{graphics}
\usepackage{lscape}
\usepackage{enumerate}
\usepackage{array}
\usepackage{framed}
\usepackage{mathtools}
\usepackage{epsfig,color}
\usepackage{hyperref}
\usepackage{tikz}
\hypersetup{colorlinks}

\usepackage{color} 

\definecolor{darkred}{rgb}{1,0,0} 
\definecolor{darkgreen}{rgb}{0,1,0}
\definecolor{darkblue}{rgb}{0,0,1}

\hypersetup{colorlinks,
linkcolor=darkblue,
filecolor=darkgreen,
urlcolor=darkred,
citecolor=darkgreen}

\theoremstyle{plain}  
\newtheorem{theorem}{Theorem}[section]
\newtheorem*{theoremA}{Theorem A}
\newtheorem*{theoremB}{Theorem B}
\newtheorem*{theoremC}{Theorem C}

\newtheorem*{theorem*}{Theorem}

\newtheorem{corollary}[theorem]{Corollary}

\newtheorem{proposition}[theorem]{Proposition}

\newtheorem{tech-lemma}[theorem]{Technical Lemma}
\newtheorem{definition}[theorem]{Definition}
\theoremstyle{remark}
\newtheorem{example}[theorem]{Example}

\newtheorem{remark}[theorem]{Remark}
\newtheorem*{remark*}{Remark}

\newtheorem*{claim*}{Claim}

\newtheoremstyle{TheoremForIntro} 
        {.6em}{.6em}              
        {\itshape}                      
        {}                              
        {\bfseries}                     
        {.}                             
        { }                             
        {\thmname{#1}\thmnote{ \bfseries #3}}
    \theoremstyle{TheoremForIntro}

\numberwithin{equation}{section}
\renewcommand{\leq}{\leqslant}

\renewcommand{\geq}{\geqslant}

\newcommand{\R}{\mathbb{R}}

\newcommand{\Z}{\mathbb{Z}}
\newcommand{\C}{\mathbb{C}}

\newcommand{\cE}{{\mathcal E}}

\newcommand{\cM}{{\mathcal M}}
\newcommand{\cN}{{\mathcal N}}
\newcommand{\cO}{{\mathcal O}}

\newcommand{\cR}{{\mathcal R}}

\newcommand{\rPSL}{\mathrm{PSL}}

\newcommand{\rSU}{\mathrm{SU}}

\newcommand{\rA}{\mathrm{A}}
\newcommand{\rB}{\mathrm{B}}

\newcommand{\rS}{\mathrm{S}}
\newcommand{\rK}{\mathrm{K}}
\newcommand{\rF}{\mathrm{F}}
\newcommand{\rE}{\mathrm{E}}
\newcommand{\rT}{\mathrm{T}}

\newcommand{\rP}{\mathrm{P}}
\newcommand{\rG}{\mathrm{G}}
\newcommand{\rC}{\mathrm{C}}
\newcommand{\rGL}{\mathrm{GL}}
\newcommand{\rL}{\mathrm{L}}

\newcommand{\rSL}{\mathrm{SL}}
\newcommand{\rSO}{\mathrm{SO}}
\newcommand{\rO}{\mathrm{O}}
\newcommand{\rU}{\mathrm{U}}

\newcommand{\fgl}{\mathfrak{gl}}
\newcommand{\fsl}{\mathfrak{sl}}

\newcommand{\fso}{\mathfrak{so}}
\newcommand{\fsp}{\mathfrak{sp}}

\newcommand{\fc}{\mathfrak{c}}

\newcommand{\fg}{\mathfrak{g}}
\newcommand{\fh}{\mathfrak{h}}
\newcommand{\fk}{\mathfrak{k}}
\newcommand{\fl}{\mathfrak{l}}
\newcommand{\fm}{\mathfrak{m}}

\newcommand{\fp}{\mathfrak{p}}
\newcommand{\fq}{\mathfrak{q}}

\newcommand{\ft}{\mathfrak{t}}
\newcommand{\fu}{\mathfrak{u}}

\newcommand{\fz}{\mathfrak{z}}

\DeclareMathOperator{\ad}{ad}
\DeclareMathOperator{\Ad}{Ad}

\DeclareMathOperator{\tr}{tr}
\DeclareMathOperator{\rk}{rk}

\DeclareMathOperator{\Hom}{Hom}
\DeclareMathOperator{\End}{End}

\DeclareMathOperator{\Id}{Id}

\newcommand{\rH}{\mathrm{H}}

\newcommand{\smtrx}[1]{\left (\begin{smallmatrix}#1\end{smallmatrix}\right)}

\let\oldmarginpar\marginpar
\renewcommand\marginpar[1]{\oldmarginpar{\tiny\bf\begin{flushleft} #1
\end{flushleft}}}

\begin{document}

\title[Arakelov--Milnor inequalities and maximal VHS]
{Arakelov--Milnor inequalities and maximal variations of Hodge structure}

\author[Biquard]{Olivier Biquard}
\address{Sorbonne Universit\'e and Universit\'e de Paris, CNRS, IMJ-PRG, F-75006 Paris, France} 
\email{olivier.biquard@sorbonne-universite.fr}

\author[Collier]{Brian Collier}
\address{900 University Ave\\ Riverside, California 92521\\
USA}
 \email{brian.collier@ucr.edu}

\author[Garc{\'\i}a-Prada]{Oscar Garc{\'\i}a-Prada}
\address{Instituto de Ciencias Matem\'aticas \\
 Nicol\'as Cabrera, 13--15 \\ 28049 Madrid \\ Spain}
\email{oscar.garcia-prada@icmat.es}

\author[Toledo] {Domingo Toledo}
\address{ Department of Mathematics \\
University of Utah\\
Salt Lake City, UT 84112 }
\email{toledo@math.utah.edu}

\thanks{
The second author was partially supported by supported by the NSF under Award No. 1604263 and NSF grants DMS-1107452, 1107263 and 1107367 “RNMS: GEometric structures And Representation varieties” (the GEAR Network). The third author was partially supported by the Spanish MINECO under the 
ICMAT Severo Ochoa  grant No.SEV-2015-0554, and grant No. MTM2016-81048-P. 
}


\begin{abstract}
In this paper we study the $\C^*$-fixed points in moduli spaces of Higgs bundles over a compact Riemann surface for a complex semisimple Lie group and its real forms. These fixed points are called Hodge bundles and correspond to complex variations of Hodge structure. 
We introduce a topological invariant for Hodge bundles that generalizes the Toledo invariant appearing for Hermitian Lie groups. 
A main result of this paper is a bound on this invariant which generalizes both the Milnor--Wood inequality of the Hermitian case and the Arakelov inequalities of classical variations of Hodge structure.
When the generalized Toledo invariant is maximal, we establish rigidity results for the associated variations of Hodge structure which generalize known rigidity results for maximal Higgs bundles and their associated maximal representations in the Hermitian case.  
\end{abstract}

\maketitle





\section{Introduction}
Since their introduction in Hitchin's seminal paper \cite{selfduality}, Higgs bundles over a 
compact Riemann surface have been of tremendous interest in geometry, topology and theoretical physics.
Within the moduli space of Higgs bundles there is a special subvariety determined by the fixed points of a natural $\C^*$-action. These fixed points are called Hodge bundles and correspond to holonomies of complex variations of Hodge structure. 
They are part of the global nilpotent cone, and coincide with critical points of a natural energy function on the moduli space.
Another importance of the $\C^*$-fixed points stems from the fact that, roughly speaking, the subvariety of Hodge bundles determines the topology of the moduli space of Higgs bundles
(see \cite{selfduality,GPHSMotivesHiggs,GPHygenus,GothenBettirk3}). In this paper, 
we investigate some basic properties of Hodge bundles and their moduli. 

To describe our results, let $\rG$ be a complex semisimple Lie group with Lie algebra $\fg$ and $X$ be a compact Riemann surface with genus $g\geq 2$ and canonical bundle $K.$ 
A $\rG$-Higgs bundle on $X$ is a pair $(E,\varphi)$, where $E$ is a holomorphic principal $\rG$-bundle on $X$ and $\varphi\in H^0(E(\fg)\otimes K)$ is a holomorphic section of the Lie algebra bundle twisted by $K.$ 
The moduli space of polystable $\rG$-Higgs bundles will be denoted by $\cM(\rG)$. By the nonabelian Hodge correspondence, the moduli space $\cM(\rG)$ is homeomorphic to the character variety $\cR(\rG)$ of conjugacy classes of reductive representations $\rho:\pi_1(X)\to\rG.$ 

The $\C^*$-action on the Higgs bundle moduli space is defined by $\lambda\cdot(E,\varphi)=(E,\lambda\varphi)$. 
To describe the $\C^*$-fixed points, fix a $\Z$-grading $\fg=\bigoplus_{j\in\Z}\fg_j$. Let $\zeta\in\fg_0$ be the grading element, i.e., $[\zeta,x]=jx$ for all $x\in\fg_j$, and let $\rG_0<\rG$ be the centralizer of $\zeta.$ A Higgs bundle $(E,\varphi)$ is said to be a Hodge bundle of type $(\rG_0,\fg_k)$ if $E$ reduces to a holomorphic $\rG_0$-bundle $E_{\rG_0},$ and $\varphi\in H^0(E_{\rG_0}(\fg_k)\otimes K).$ 
A polystable Higgs bundles is a fixed point if and only if it is a Hodge bundle \cite{localsystems}. In fact it suffices to consider Hodge bundles of type $(\rG_0,\fg_1)$ (see \S\ref{sec moduli fixed} for more details). 

The representations $\rho:\pi_1(X)\to\rG$ associated to Hodge bundles factor through a real form $\rho:\pi_1(X)\to\rG^\R\to\rG$ of Hodge type, i.e., a real form such that the rank of its maximal compact subgroup is equal to the rank of $\rG.$ 
If $\rH_0^\R<\rG_0$ is a maximal compact subgroup, then the homogeneous space $\rG^\R/\rH_0^\R$ has a natural complex structure and is called a period domain. The holomorphic tangent bundle of $\rG^\R/\rH_0^\R$ decomposes as $\bigoplus_{j\geq 0}\rG^\R\times_{\rH_0^\R}\fg_j$. 
In \cite{localsystems}, Simpson showed that the representations $\rho:\pi_1(X)\to\rG^\R<\rG$ arising from Hodge bundles of type $(\rG_0,\fg_1)$
define $\rho$-equivariant holomorphic maps
\[f_\rho:\widetilde X\to\rG^\R/\rH_0^\R\]
such that $\partial f_\rho$ is valued in the first graded piece $\rG^\R\times_{\rH_0^\R}\fg_1$ of the holomorphic tangent bundle. Such pairs $(\rho,f_\rho)$ are called variations of Hodge structure. 

When the grading in the above discussion is $\fg=\fg_{-1}\oplus\fg_0\oplus\fg_1,$ the real form $\rG^\R<\rG$ is a group of Hermitian type and the period domain $\rG^\R/\rH_0^\R$ is the Riemannian symmetric space of $\rG^\R.$ 
In this case, the symmetric space is K\"ahler and Hodge bundles define representations $\rho:\pi_1(X)\to\rG^\R$ and $\rho$-equivariant holomorphic maps to the symmetric space of $\rG^\R.$ 
Pulling back the K\"ahler form and integrating it over $X$ defines an invariant which is usually called the Toledo invariant $\tau(\rho)$. 
The Toledo invariant is defined for all representations into Hermitian Lie groups and satisfies the Milnor--Wood inequality
\[|\tau(\rho)|\leq (2g-2)\rk(\rG^\R/\rH^\R).\]
This inequality was first proven by Milnor \cite{MilnorMWinequality} for $\rPSL_2\R$, and more generally in \cite{MilnorWoodIneqDomicToledo,Burger-Iozzi-Wienhard-CR}. Using Higgs bundles, one can also define the Toledo invariant 
and obtain the Milnor--Wood inequality \cite{selfduality,UpqHiggs,BGRmaximalToledo}.

The set of representations which maximize the invariant are called maximal representations and have many interesting geometric features. 
For example, maximal representations define connected components of the character variety which consist entirely of discrete and faithful representations \cite{BIWmaximalToledoAnnals}. For $\rPSL_2\R$, maximal representations correspond to Fuchsian representations and the representation uniformizing the Riemann surface $X$ defines the unique maximal Hodge bundle.
There are two classes of Hermitian groups, those of tube type and those of nontube type. 
Generalizing the complex hyperbolic geometry results of \cite{ToledoSU(1n)Rigid}, maximal representations into nontube type groups always factor through a maximal subtube up to a compact factor \cite{UpqHiggs,Burger-Iozzi-Wienhard-CR,BGRmaximalToledo}. For example $\rSU(p,p)$ is the maximal subtube of $\rSU(p,q)$ when $q\geq p$. 

Similar to the Milnor--Wood inequality, in a somewhat different context, one has the Arakelov-type  inequalities for classical variations of 
Hodge structure.  These are generalizations of the classical Arakelov inequality for the degree of the relative canonical bundle of a family of curves over another curve (see, for example \cite{PetersRigidityVHSArakelov,Jost-Zuo,viehweg,ViehwegZuo,MollerViehwegZuo}).

Hodge bundles of type $(\rG_0,\fg_1)$ form their own moduli space $\cM(\rG_0,\fg_1)$. 
Consider a $\Z$-grading $\fg=\bigoplus_{j\in\Z}\fg_j$ with grading element $\zeta$ and define a character 
\[\chi_T:\fg_0\to\C~,~ \chi_T(x)=B(x,\zeta)B(\gamma,\gamma),\]
where $B$ is the Killing form of $\fg$ and $\gamma$ is the longest root such that the root space $\fg_\gamma\subset\fg_1.$ 
As explained in \S\ref{sec hol sec curv}, the normalization constant $B(\gamma,\gamma)$ normalizes the minimum of the holomorphic sectional curvature of the period domain to be $-1$. 
Recall that a Hodge bundle of type $(\rG_0,\fg_1)$ is a pair $(E,\varphi)$, where $E$ is a holomorphic $\rG_0$-bundle and $\varphi\in H^0(E(\fg_1)\otimes K).$ A multiple $q\chi_T$ for some positive integer $q$ lifts to a character $\chi:\rG_0\to\C^*$ and defines a line bundle $E(\chi).$ We define the Toledo invariant of a Hodge bundle $(E,\varphi)$ by
\[\tau(E,\varphi)=\frac{1}{q}\deg E(\chi).\]
For gradings $\fg_{-1}\oplus\fg_0\oplus\fg_1$, the character $\chi_T$ and the invariant $\tau$ agree with the definition of the Toledo character and the Toledo invariant for Hermitian groups in \cite{BGRmaximalToledo}. 

To generalize the notion of rank, we use the fact that the space $\fg_1$ is a prehomogeneous vector space for the action of $\rG_0.$ That is, $\fg_1$ has a unique open dense $\rG_0$-orbit $\Omega\subset\fg_1.$ 
The theory of prehomogeneous vector spaces was introduced by Sato (see \cite{SatoKimuraPHVS,KimuraIntroPHVS,MortajineRedbook,ManievelPrehom}) and provides an ideal set of tools to study our problem. Let $e\in \Omega\subset\fg_1$ be any point in the open orbit and complete it to an $\fsl_2$-triple $\{f,h,e\}$ with $h\in\fg_0.$ We define the rank of $(\rG_0,\fg_1)$ by 
\[\rk_T(\rG_0,\fg_1)=\frac{1}{2}\chi_T(h).\]
In \S \ref{sec toledo char and per dom} we show that this is independent of the choices made. Again, for Hermitian groups this definition recovers the rank of the symmetric space. 

Our first main result is:
\begin{theoremA}[Corollary \ref{cor am ineq for comp}]\label{theoremA}
  A polystable Hodge bundle $(E,\varphi)$ of type $(\rG_0,\fg_1)$ satisfies the inequality
\begin{equation}
  \label{eq intro am ineq}|\tau(E,\varphi)|\leq (2g-2)\rk_T(\rG_0,\fg_1).
\end{equation}
\end{theoremA}
We refer to the inequality \eqref{eq intro am ineq} as the Arakelov--Milnor inequality as it generalizes both the Milnor--Wood inequality and the Arakelov inequalities for variations of Hodge structure. We refer to Hodge bundles with $|\tau(E,\varphi)|=\rk_T(\rG_0,\fg_1)(2g-2)$ as maximal Hodge bundles of type $(\rG_0,\fg_1)$. As in the Hermitian case of \cite{BGRmaximalToledo}, the Arakelov--Milnor inequality follows from a more general and more precise inequality established in Theorem \ref{thm: Arakelov--Milnor ineq}. Indeed,  the stability of a Hodge bundle of type $(\rG_0,\fg_1)$ depends on a  parameter $\alpha=\lambda\zeta$ with $\lambda\in \R$. In Theorem \ref{thm: Arakelov--Milnor ineq} we give an inequality for an $\alpha$-semistable Hodge bundle 
$(E,\varphi)$. Moreover,
we give a more refined inequality  which depends on the orbit in $\fg_1$ which contains the generic value of $\varphi$. While we are mostly interested  in the case $\alpha=0$, since this relates to the stability of the Higgs bundle 
obtained from $(E,\varphi)$ by extension of structure group, considering the 
arbitrary value of $\alpha$ case has proven to be a powerful tool  to study the 
moduli space for $\alpha=0$  
(see \cite{UpqHiggs,GPHSMotivesHiggs}).
We believe that the same principle will apply in this general situation.

\begin{example}\label{running_example}
  The inequality is usually easy to write down in concrete cases. For example, for $\rG=\rSO_{2p+q}\C$ and $\rG_0=\rGL_p\C\times \rSO_q\C$, one finds that $\rk_T(\rG_0,\fg_1)=2\min(p,q)$ if $q>1$ and $1$ if $q=1$.
Let us take $q>1$. A bundle $(E,\varphi)$ of type $(\rG_0,\fg_1)$ can be written as $E=V\oplus W\oplus V^*$, where $V$ is a $\rGL_p\C$-bundle and $W$ is a $\rSO_q\C$-bundle, and the Higgs field has the form
\[ \varphi =\left(
  \begin{smallmatrix}
    0 & \theta & 0 \\ 0 & 0 & - \theta^T \\ 0 & 0 & 0 
  \end{smallmatrix}\right)\]
where $\theta:W\rightarrow V\otimes K$. One obtains $\tau(E,\varphi)=2 \deg V$, and it follows that the inequality (\ref{eq intro am ineq}) takes the form (actually $\deg V\leq0$)
\[ \deg V \geq - \min(p,q) (2g-2). \]
As mentioned above, we actually prove a more precise inequality, which depends on the orbit which contains the generic value of $\varphi$. In our case, if the image of $\theta^T$ in $W$ is the sum of a nondegenerate subspace of dimension $r_1$ and of a totally isotropic subspace of dimension $r_2$, then the inequality can be refined in
\begin{equation}
 \deg V \geq - (r_1+\tfrac{r_2}2) (2g-2).\label{eq:4}
\end{equation}
These values are calculated in Example \ref{calc_SO2pq}.
\end{example}
Fundamental in the proof of the Arakelov--Milnor inequality \eqref{eq intro am ineq} is the  construction of a maximal Jacobson--Morozov regular prehomogeneous vector subspace of $(\rG_0,\fg_1)$ and the existence of a relative invariant for this subspace. This is the analogue of a maximal subtube in the Hermitian case.  
To explain this, let $\{f,h,e\}\subset\fg$ be an $\fsl_2$-triple such that $[h,e]=2e$ and $[e,f]=h.$ 
The weights of $\ad_h$ are integral. When the weights of $\ad_h$ are all even, the $\{f,h,e\}$ is called an even $\fsl_2$-triple and defines a $\Z$-grading $\fg=\bigoplus_{j\in\Z}\fg_j$ with grading element $\frac{h}{2}$ and $e\in\fg_1.$ In fact $e\in\Omega\subset\fg_1$ is in the open $\rG_0$-orbit (see \S\ref{sec phvs for Zgrad}). 
A prehomogeneous vector space $(\rG_0,\fg_1)$ is called regular if the $\rG_0$-stabilizer of a point in the open orbit $\Omega$ is reductive. 
For gradings arising from even $\fsl_2$-triples $\{f,h,e\}$, $(\rG_0,\fg_1)$ is regular since the $\rG_0$-centralizer of $e\in\Omega$ coincides with the $\rG$-centralizer of the $\fsl_2\C$-subalgebra. In this case, we refer to $(\rG_0,\fg_1)$ as a JM-regular prehomogeneous vector space. 

Fix a basis $\{f,h,e\}$ of $\fsl_2\C$ and let $\rT<\rPSL_2\C$ be the connected subgroup with Lie algebra $\langle h\rangle.$ The Hodge bundle associated to the representation uniformizing the Riemann surface $X$ is given by $(E_\rT,e)$, where $E_\rT$ is the holomorphic frame bundle of $K^{-1}$ and $e\in H^0(E_\rT(\langle e\rangle)\otimes K).$ For $\rSL_2\C,$ we take $E_\rT$ to be the frame bundle of a square root of $K^{-1}.$
Suppose $\{f,h,e\}\subset\fg$ is an even $\fsl_2$-triple and consider the associated JM-regular phvs $(\rG_0,\fg_1).$
From our set up, it follows that extending structure group $(E_\rT(\rG_0),e)$ defines a maximal Hodge bundle, i.e.,
\[\tau(E_\rT(\rG_0),e)=\rk_T(\rG_0,\fg_1)(2g-2).\]
If $\rC<\rG_0$ is the $\rG$-centralizer of the $\fsl_2\C$-subalgebra. Since $\rT$ and $\rC$ are commuting subgroups of $\rG_0$, we can form $\rG_0$-bundle $E_\rT\otimes E_\rC(\rG_0)$ out of $E_\rT$ and a holomorphic $\rC$-bundle $E_\rC$. Since $\rC$ acts trivially on $\langle e\rangle,$ we have $e\in H^0(E_\rC\otimes \rT(\langle e\rangle)\otimes K)$. Moreover, this process does not change the Toledo invariant 
and defines a map 
\[\Psi_e:\xymatrix@R=0em{\cN(\rC)\ar[r]&\cM^{\max}(\rG_0,\fg_1)\\E_\rC\ar@{|->}[r]&(E_\rC\otimes E_\rT(\rG_0),e)}\]
from the moduli space of degree zero polystable $\rC$-bundles to the moduli space of maximal Hodge bundles of type $(\rG_0,\fg_1).$
\begin{theoremB}(Theorem \ref{thm rigidity jm regular})
   If $(\rG_0,\fg_1)$ is a JM-regular prehomogeneous vector space, then the map $\Psi_e$ defines an isomorphism between the moduli space of degree zero polystable $\rC$-bundles and the moduli space of maximal Hodge bundles of type $(\rG_0,\fg_1)$.
 \end{theoremB} 
 \begin{remark}
   We note that the map $\Psi_e$ is a moduli space version of a restriction of the so called global Slodowy slice map of \cite{ColSandGlobalSlodowy}. In \cite{MagicalBCGGO}, an extension of the map $\Psi_e$ to the entire Slodowy slice gives rise to the Cayley correspondence  used to describe certain components of the moduli space of $\rG^\R$-Higgs bundles related to higher Teichm\"uller theory, generalizing the Cayley correspondence for  the Hermitian group case given   in \cite{BGRmaximalToledo}. 
   It would be interesting to extend these results to the full Slodowy slice for every even $\fsl_2$-triple. 
 \end{remark}
When $(\rG_0,\fg_1)$ is JM-regular, the representations $\rho:\pi_1(X)\to\rG$ associated to maximal Hodge bundles of type $(\rG_0,\fg_1)$ are described by the following theorem.
 \begin{theoremC}
   (Theorem \ref{thm jm reg rigidity of reps for maximal}) Fix a Riemann surface $X$ of genus $g\geq 2$ and suppose $(\rG_0,\fg_1)$ is a JM-regular prehomogeneous vector space associated to an even $\fsl_2$-triple $\{f,h,e\}$. Let $\rS<\rG$ be the associated connected subgroup and let $\rC$ be the $\rG$-centralizer of $\{f,h,e\}$. The Higgs bundle associated to a reductive representation $\rho:\pi_1(X)\to\rG$ is a maximal Hodge bundle of type $(\rG_0,\fg_1)$ if and only if $\rho$ is a product $\rho=\rho_u*\rho_\rC$, where
   \begin{itemize}
     \item $\rho_u:\pi_1(X)\to\rS^\R<\rS$ is the uniformizing $\rPSL_2\R$-representation of $X$ if $\rS\cong\rPSL_2\C$ and a lift of the uniformizing representation to $\rSL_2\R$ if $\rS\cong\rSL_2\C,$ and 
     \item $\rho_\rC:\pi_1(X)\to\rC^\R<\rC$ is any representation into the compact real form of $\rC.$
   \end{itemize}
 \end{theoremC}
The representations $\rho=\rho_u*\rho_\rC:\pi_1(X)\to\rG$ all factor through a real form of Hodge type $\rG^\R<\rG$ which is canonically associated to the grading $\bigoplus_{j\in\Z}\fg_j$ (see Proposition \ref{prop hodge real form of Zgrad}). Moreover, the $\rG^\R$-centralizer of any such representation is compact, and hence these representations do not factor through any proper parabolic subgroups $\rP^\R<\rG^\R.$

From the above results, it follows that the equivariant holomorphic map $f_\rho:\widetilde X\to\rG^\R/\rH_0^\R$ associated to such a maximal variation of Hodge structure $(\rho,f_\rho)$ is a totally geodesic embedding which maximizes the holomorphic sectional curvature. In the paper  we discuss how the choice of the Toledo character is related to a metric of minimal holomorphic
sectional curvature $-1$ on $\rG^\R/\rH_0^\R$, providing an alternative proof of Theorem
\ref{theoremA}. Our bounds for the sectional curvature give the bounds found 
in \cite{Qiongling_Li} in the case of  $\rSL_n\C$.

 Now consider a general grading $\fg=\bigoplus_{j\in\Z}\fg_j$ with grading element $\zeta$, i.e., not necessarily coming from an $\fsl_2$-triple. Pick a point $e\in\Omega\subset\fg_1$ in the open $\rG_0$-orbit, and $e$ to an $\fsl_2$-triple $\{f,h,e\}$ with $h\in\fg_0.$ If $(\rG_0,\fg_1)$ is not a JM-regular phvs, then $s=\zeta-\frac{h}{2}$ is nonzero and we define $\hat\rG_0<\rG_0$ to be the $\rG_0$ centralizer of $s$ and $\hat \fg_1=\{x\in\fg_1~|~[s,x]\}=0.$  
 With this set up, $\hat\rG_0$ is reductive and $(\hat\rG_0,\hat\fg_1)$ defines a prehomogeneous vector subspace of $(\rG_0,\fg_1)$ which contains $e\in\hat\Omega\subset\hat\fg_1$ is the open orbit (see \S \ref{sec JM reg phvss}). 
 Moreover, the $\rG_0$-centralizer $\hat\rC$ of the $\fsl_2$-triple $\{f,h,e\}$ satisfies $\hat\rC<\hat\rG_0.$ We call $(\hat\rG_0,\hat\fg_1)$ a maximal JM-regular prehomogeneous vector subspace of $(\rG_0,\fg_1).$

In \S \ref{sec non JM reg max Hodge}, we show that all maximal Hodge bundles of type $(\rG_0,\fg_1)$ reduce to maximal Hodge bundles of type $(\rG_0,\fg_1)$ (see Proposition \ref{prop no stable reduce to max JMreg}). 
We then show that the moduli space $\cM^{\max}(\rG_0,\fg_1)$ of maximal Hodge bundles of type $(\rG_0,\fg_1)$ is isomorphic to the moduli space of polystable $\hat\rC$-bundles and deduce rigidity results for maximal variations of Hodge structure analogous (see Theorems \ref{thm param of nonJM reg max} and \ref{rigidity non JM reg case}). 
This recovers rigidity results for maximal variations of Hodge structure in the Hermitian case. 

\noindent
{\bf Acknowledgements:} We wish to thank Nigel Hitchin and Kang Zuo for very 
useful discussions. We also want to thank the Institut Henri Poincar\'e for 
support under the RIP programme in 2017.  

\section{Prehomogeneous vector spaces, $\Z$-gradings and the Toledo character }
\label{sec phvs and all that}
For this section, let $\rG$ be a complex reductive Lie group with Lie algebra $\fg.$ 

\subsection{Prehomogeneous vector spaces}
We collect some basic facts about prehomogeneous vector spaces. Main references are \cite{KimuraIntroPHVS,knappbeyondintro,MortajineRedbook,SatoKimuraPHVS}.

A {\bf prehomogeneous vector space (phvs)} for $\rG$ is a finite dimensional complex vector space $V$ together with a holomorphic representation $\rho:\rG\to\rGL(V)$ such that $V$ has an open $\rG$-orbit. Such an open orbit is necessarily unique and dense. If $V$ is a phvs, let $\Omega$ denote the {\bf open orbit} in $V$ and $S=V\setminus\Omega$ be the {\bf singular set}. 
For $x\in V$, denote the $\rG$-stabilizer of $x$ by $\rG^x.$ A phvs vector space $V$ is called {\bf regular} if $\rG^x$ is reductive for $x\in\Omega$, otherwise it is called {\bf nonregular}.  

We say that  $(\rH,W)$ is a  
{\bf prehomogeneous vector subspace (phvss)} of $(\rG,V)$  if $(\rH,W)$ is a phvs, $\rH\subset \rG$ is a subgroup, $W\subset V$ is a vector
subspace, and the action of $\rH$ is the restriction of the action of $\rG$.

\begin{example}\label{ex phvs M_p,q} (1) The vector space $\C^n$ is a phvs for the standard representation of $\rGL_n\C$. For this example, $\Omega=\C^n\setminus\{0\}$, and it is regular only when $n=1.$

\noindent (2) The vector space $M_{p,q}$ of $p\times q$-matrices is a phvs for the action of $\rS(\rGL_p\C\times\rGL_q\C)$ given by $(A,B)\cdot M= AMB^{-1}.$ Here, $\Omega=\{M\in M_{p,q}~|~ \rk(M)=\min(p,q)\}.$ This example is regular only when $p=q$. 

\noindent (3) The vector space $M_{p,q}$ is also a phvs for the action of $\rGL_p\C\times\rSO_q\C$ given by $(A,B)\cdot M=AMB^{-1}.$ Here $\Omega=\{ M\in M_{p,q}~|~\rk(M\cdot M^T)=\min(p,q)\}$. Also, the vector space $M_{p,q}\oplus M_{q,r}$ is a phvs for the action of $\rS(\rGL_p\C\times\rGL_q\C\times\rGL_r\C)$ given by $(A,B,C)\cdot (M,N)=(AMB^{-1},BNC^{-1}).$ Here $\Omega=\{(M,N)\in M_{p,q}\oplus M_{q,r}~|~\rk(MN)=\min(p,q,r)\}$. 
The first example is regular when $p\leq q$ and the second example is regular when $p=r$ and $p\leq q$.
Note that the inclusions
\[\xymatrix@R=0em{M_{p,q}\ar[r]& M_{p,q}\oplus M_{q,p}&&\rGL_p\C\times\rSO_q\C\ar[r]&\rS(\rGL_p\C\times\rGL_q\C\times\rGL_p\C)\\M\ar@{|->}[r]&(M,-M^T)&&(A,B)\ar@{|->}[r]&(A,B,(A^T)^{-1})}\] 
makes $(\rGL_p\C\times\rSO_q\C,M_{p,q})$ a phvss of $(\rS(\rGL_p\C\times\rGL_q\C\times\rGL_p\C),M_{p,q}\oplus M_{q,p}).$
\end{example}

Let $V$ be a phvs for $\rG$ with representation $\rho.$ A non-constant rational function $F:V\to\C$ is called a {\bf relative invariant} if there exists a character $\chi:\rG\to\C^*$ so that 
\[F(\rho(g)\cdot x)=\chi(g)F(x)~\text{for all $g\in\rG$ and $x\in V$}.\]
Here are some fundamental facts appearing in \cite[\S4]{SatoKimuraPHVS}. 
\begin{proposition}\label{prop: phvs facts}
  Let $V$ be a phvs for $\rG$ with representation $\rho$.
  \begin{enumerate}
     \item Up to a constant, a relative invariant is uniquely determined by its corresponding character. In particular, any relative invariant is a homogeneous function.
     \item $V$ is regular if and only if the singular set $S$ is a hypersurface.
     \item Let $\chi:\rG\to\C^*$ be a character. Then there is a relative invariant for $\chi$ if and only if $\chi$ is trivial on the stabilizers of points in $\Omega,$ i.e., $\chi|_{\rG^x}=1$ for all $x\in\Omega.$
   \end{enumerate} 
\end{proposition}

\begin{example}
  The regular phvs $M_{p,p}$ from Example \ref{ex phvs M_p,q} part (2) has a  relative invariant $F:M_{p,p}\to\C$ given by $F(M)=\det(M).$ The associated character $\chi:\rG\to\C^*$ is given by $\chi(A,B)=\det(A)\det(B)^{-1}$ since 
  \[F((A,B)\cdot M)=\det(AMB^{-1})=\chi(A,B)F(M).\]

  The regular phvs $M_{p,q}\oplus M_{q,p}$ with $p\leq q$ from Example \ref{ex phvs M_p,q} part (3) has a relative-invariant given by $F(M,N)=\det(MN)$. The associated character is $\chi(A,B,C)=\det(A)\det(C^{-1})$ since 
  \[F((A,B,C)\cdot(M,N))=\det(AMB^{-1}BNC^{-1})=\chi(A,B,C)F(M,N).\] 
  This relative invariant also defines a relative invariant for the the $\rGL_p\C\times \rSO_q\C$ phvss given by $(M,N)=(M,-M^T)$. For $(A,B)\in\rGL_p\C\times \rSO_q\C$, the associated character is $\chi(A,B)=\det(AA^T)=\det(A)^2.$
\end{example}

\subsection{Prehomogeneous vector spaces associated to $\Z$-gradings}\label{sec phvs for Zgrad}
A $\Z$-grading of a semisimple Lie algebra $\fg$ is a decomposition 
\[\fg=\bigoplus_{j\in\Z}\fg_j\ \ \ \ \  \text{such that }\ \ \ \ \ [\fg_i,\fg_j]\subset\fg_{i+j}.\] The subalgebra $\fp=\bigoplus_{j\geq 0}\fg_j$ is a parabolic subalgebra with Levi subalgebra $\fg_0\subset\fp.$ There is an element $\zeta\in\fg_0$ such that $\fg_j=\{X\in\fg~|~[\zeta,x]=jx\}$; the element $\zeta$ is called the {\bf grading element} of the $\Z$-grading.

Given a $\Z$-grading $\fg=\bigoplus_{j\in\Z}\fg_j$, let $\rG_0<\rG$ be the centralizer of $\zeta$; $\rG_0$ acts on each factor $\fg_j$. The relation between $\Z$-gradings and prehomogeneous vector spaces is given by the following theorem of Vinberg (see \cite[Theorem 10.19]{knappbeyondintro}).

\begin{proposition}\label{prop each gj phvs}For each $j\neq 0,$ $\fg_j$ is a prehomogeneous vector space for $\rG_0.$
 \end{proposition} 
Prehomogeneous vector spaces arising from $\Z$-gradings of $\fg$ are said to be of {\bf parabolic type}. From now on, we will consider only phvs's of parabolic type. 
\begin{example}\label{ex parabolic type}
  All of the phvs's from Example \ref{ex phvs M_p,q} are of parabolic type.
  \begin{itemize}
     \item For $(\rS(\rGL_p\C\times\rGL_q\C),M_{p,q})$, $\rG=\rSL_{p+q}\C$ and the parabolic is the stabilizer a $p$-plane in $\C^{p+q}.$
     \item For $(\rGL_p\C\times\rSO_q\C,M_{p,q})$, $\rG=\rSO_{2p+q}\C$ and the parabolic is the stabilizer of an isotropic $p$-plane in $\C^{2p+q}$.
     \item For $(\rS(\rGL_p\C\times\rGL_q\C\times\rGL_r\C),M_{p,q}\oplus M_{q,r})$, $\rG=\rSL_{p+q+r}\C$ and the parabolic stabilizer of a flag $\C^p\subset\C^{p+q}\subset \C^{p+q+r}.$
   \end{itemize}  
\end{example}

Consider a nonzero nilpotent element $e\in\fg.$ By the Jacobson--Morozov theorem, $e$ can be completed to an $\fsl_2$-triple $\{f,h,e\}\subset \fg$. That is, a triple satisfying the bracket relations of $\fsl_2\C:$
\[\xymatrix{[h,e]=2e~,&[h,f]=-2f&\text{and}&[e,f]=h}.\] 
Moreover, given $\{h,e\}$ so that $[h,e]=2e$ and $h\in\ad_e(\fg)$, there is a unique $f\in\fg$ so that $\{f,h,e\}$ is an $\fsl_2$-triple.

Given an $\fsl_2$-triple $\{f,h,e\}$, the semisimple element $h$ acts on $\fg$ with integral weights and thus defines a $\Z$-grading $\fg=\bigoplus_{j\in\Z}\fg_j$, where $\fg_j=\{x\in\fg~|~\ad_h(x)=jx\}.$ Note that $e\in\fg_2.$ We have the following result of Kostant and Malcev (see \cite[Theorem 10.10]{knappbeyondintro}).

\begin{proposition}
  Let $\{f,h,e\}\subset\fg$ be an $\fsl_2$-triple with associated $\Z$-grading $\fg=\bigoplus_{j\in\Z}\fg_j$, and let $\rG_0<\rG$ be the associated analytic subgroup with Lie algebra $\fg_0.$ Then $e$ is in the open orbit $\Omega$ of the phvs $(G_0,\fg_2)$. 
\end{proposition}
\begin{corollary}\label{cor: JM reg}
  The phvs $(\rG_0,\fg_2)$ arising from an $\fsl_2$-triple $\{f,h,e\}$ is regular.
\end{corollary}
\begin{proof}
  Since $e\in\Omega\subset\fg_2$, we need to show that the stabilizer $\rG_0^e$ is reductive. The group $\rG_0$ centralizes $h,$ so the stabilizer $\rG_0^e$ of $e$ stabilizes both $h$ and $e$. By the uniqueness part of the Jacobson--Morozov theorem, $\rG_0^e$ stabilizes the $\fsl_2$-triple. Since the centralizer of a reductive subalgebra is reductive, we conclude that $\rG_0^e$ is reductive.
\end{proof}
Let $B:\fg\times\fg\to\C$ denote the Killing form of $\fg.$ Given an $\fsl_2$-triple, $\{f,h,e\}\subset\fg$ with associated $\Z$-grading $\fg=\bigoplus_{j\in\Z}\fg_j$, define the character $\chi_{h}:\fg_0\to\C$ by
\begin{equation}
 \label{eq chi h} \chi_{h}(x)=\frac{1}{2}B(h,x).
\end{equation}

\begin{proposition}\label{prop: relative invariant for sl2 triple}
There is a positive integer $q$ so that $q\chi_{h}$ lifts to a character
  \[\chi_{h,q}:\rG_0\to \C^*\]
which has a polynomial relative invariant of degree $qB(\frac{h}{2},\frac{h}{2})$. 
\end{proposition}
\begin{proof}
 Let $\fg_0^e\subset\fg_0$ be the Lie algebra of the $\rG_0$-stabilizer of $e\in\fg_2.$  For $x\in\fg^e_0$, we have 
 \[B(h,x)=B([e,f],x)=B(f,[e,x])=0.\] 
Thus, $\chi_h(\fg_0^e)=0$, and there is a positive integer $q$ so that $q\chi_h$ lifts to a character which is trivial on the identity component of $\rG_0^e.$ Since $\rG_0^e$ has a finite number of components, we can choose $q$ so that $q\chi_h$ lifts to a character $\chi_{h,q}:\rG_0\to\C^*$ whose restriction to $\rG_0^e$ is trivial. Hence, by part (4) of Proposition \ref{prop: phvs facts}, the $\chi_{h,q}$ has a relative invariant $F:\rG_0\to\rC^*$ such that $F(e)\neq0$. 
For the degree, we have 
\[F(\exp(t)e)=F(\exp(\frac{h}{2}t)\cdot e)=\chi_{h,q}(\exp(\frac{h}{2}t))F(e)=\exp(tqB(\frac{h}{2},\frac{h}{2}))F(e).\]
By Proposition \ref{prop: phvs facts} $F$ is homogeneous, so $\deg(F)=kB(\frac{h}{2},\frac{h}{2}).$
\end{proof}

\begin{remark}\label{remark: not all reg are JM reg}
  Not every regular phvs of the form $(\rG_0,\fg_2)$ arises from an $\fsl_2$-triple. For examples in $\rE_6$, $\rE_7$ and $\rE_8$ see \cite[Remark 4.18]{RubenthalerNonJMRegPHVS}. In fact, the $\rB_3$ example in \cite[Table 1]{RubenthalerNonJMRegPHVS} gives a very simple example of a regular phvs which does not arise from an $\fsl_2$-triple. For such regular phvs's the grading element $\zeta$ does not have a relative invariant.
\end{remark}

\subsection{Canonical $\Z$-gradings associated to parabolics}\label{section: canonical Z grade}
In this section we fix some normalizations for the $\Z$-gradings we will consider. 

Up to conjugation, all $\Z$-gradings of $\fg$ arise from labeling the nodes of the Dynkin diagram of $\fg$ with non-negative integers \cite[Chapter 3.5]{LieIII}. This works as follows.
Let $\ft\subset\fg$ be a Cartan subalgebra and $\Delta=\Delta(\fg,\ft)$ be the corresponding set of roots. Pick a set of simple roots $\Pi\subset\Delta$ and let $\Delta^+\subset\Delta$ be the the set of positive roots. Every root $\alpha\in\Delta^+$ can be written as $\alpha=\sum_{\alpha_i\in\Pi}n_i\alpha_i$, where all $n_i$ are nonnegative  integers. 
For each $\alpha_i\in\Pi$ choose nonnegative integers $p_i$, i.e. label the nodes of the Dynkin diagram of $\Pi$ with nonnegative integers $p_i$.  This choice defines a $\Z$-grading $\fg=\bigoplus_{j\in\Z}\fg_j$, where $\fg_j$ is the direct sum of all root spaces $\fg_\alpha$ such that $\alpha=\sum_{\alpha_i\in\Pi}n_i\alpha_i$ and $j=\sum_{i}n_ip_i.$ 

In the above construction, the associated parabolic $\fp=\bigoplus_{j\geq0}\fg_j$ and the Levi subalgebra $\fg_0$ depend only on  the labels $p_i\mod 2$. Namely, up to conjugation, parabolic subalgebras $\fp$ are determined by subsets $\Theta\subset\Pi$, where
\[\fp_\Theta=\ft\oplus\bigoplus_{\alpha\in span(\Theta)\cap \Delta^+}\fg_{-\alpha}\oplus \bigoplus_{\alpha\in\Delta^+}\fg_\alpha.\]
We define the {\bf canonical $\Z$-grading of} $\fp_\Theta$ by labeling the $\alpha_i$-node of the Dynkin diagram with $0$ if $\alpha_i\in\Theta$ and $1$ if $\alpha_i\in\Pi\setminus\Theta.$ That is $\fg_j$ is given by 
\[\fg_j=\bigoplus\fg_\alpha \ , \ \ \ \text{where $\alpha=\sum_{\alpha_i\in\Pi}n_i\alpha_i$ and $j=\sum_{\alpha_i\in\Pi\setminus\Theta}n_i$}.\]

\begin{remark}
 We will often consider phvs's of parabolic type and of the form $(\rG_0,\fg_1)$ where $\fg_1$ is the weight 1 piece of the canonical grading of a parabolic subalgebra of $\fg.$ This is not a major restriction since, given a general phvs $(\rG_0,\fg_k)$ of parabolic type, the subalgebra consisting of graded pieces $\fg_j$ with $j=0\mod k$ is reductive and $\fg_k$ is the weight 1 piece of the canonical grading of a parabolic in this subalgebra.
\end{remark}

An $\fsl_2$-triple $\{f,h,e\}\subset\fg$ is called {\bf even} if $\ad_h$ has only even eigenvalues. Parabolic subalgebras arising from even $\fsl_2$-triples are called {\bf even Jacobson--Morozov parabolics}. For such parabolics, the canonical grading is given by $\ad_{\frac{h}{2}}$; the phvs $(\rG_0,\fg_1)$ is regular by Corollary \ref{cor: JM reg} and, by Proposition \ref{prop: relative invariant for sl2 triple}, there is a polynomial invariant associated to the character $\chi_h$ from \eqref{eq chi h}. 
As a result we will call a phvs $(\rG_0,\fg_1)$ which arises from an even $\fsl_2$-triple an {\bf Jacobson--Morozov regular phvs (JM-regular phvs)}.

\begin{example}\label{ex principal}
  When $\Theta=\emptyset,$ $\rG_0$ is Cartan subgroup and $\fg_1$ is the direct sum of simple root spaces. Such a phvs $(\rG_0,\fg_1)$ is always JM-regular and arises from a principal $\fsl_2$-triples. Here $\Omega\subset\fg$ consists of vectors with nonzero projection onto each simple root space and the stabilizer of such a point is the center of $\rG.$
\end{example}

\begin{example}\label{ex compact dual of herm}
Let $\eta=\sum_{\alpha_i\in\Pi}n_i\alpha_i$ be the longest root of $\Delta^+.$ For $\fg$ not of type $\rE_8,~\rF_4,~\rG_2,$ there is at least one simple root  $\alpha_i$ with $n_i=1$. For such a root set $\Theta=\Pi\setminus\{\alpha_i\}.$ In this case, the canonical $\Z$-grading of the parabolic $\fp_\Theta$ is given by
  \[\fg=\fg_{-1}\oplus\fg_0\oplus\fg_1.\]
For these examples, $(\rG_0,\fg_1)$ is regular if and only if it is JM-regular. The associated flag variety $\rG/\rP_\Theta$ is the compact dual of a Hermitian symmetric space and $(\rG_0,\fg_1)$ is regular if and only if the associated Hermitian symmetric space is of tube type. An example of this is given in Example \ref{ex phvs M_p,q}, in which case the symmetric space of $\rSU(p,q)$ is the relevant Hermitian symmetric space; it is of tube type only when $q=p$.
\end{example}

\begin{remark}\label{rem ex M_pq JM-reg}
In Example \ref{ex compact dual of herm}, the space $\fg_1$ is an irreducible $\rG_0$-representation. In general, if $(\rG_0,\fg_1)$ is a phvs such that $\rG_0$ acts irreducibly on $\fg_1,$ then $(\rG_0,\fg_1)$ is regular if and only if it is JM-regular. For example, this implies that the phvs $(\rGL_p\C\times\rSO_q\C, M_{p,q})$ from Examples \ref{ex phvs M_p,q} and \ref{ex parabolic type} is JM-regular whenever $p\leq q.$ 
The regular phvs $(\rS(\rGL_p\C\times\rGL_q\C\times\rGL_p\C),M_{p,q}\oplus M_{q,p})$ from Examples \ref{ex phvs M_p,q} and \ref{ex parabolic type} does not come from a maximal parabolic of $\rSL_{2p+q}\C$ but it is still a JM-regular phvs when $p\leq q$. 
\end{remark}

\begin{example}\label{ex reg not jm reg}
  The simplest example of a phvs $(\rG_0,\fg_1)$ which is regular but not JM-regular occurs in $\rSO_7\C$ with $\Theta=\{\alpha_2\}$, i.e., with labeled Dynkin diagram
\begin{center}
    \begin{tikzpicture}[scale=.7]        
    \draw[thick] (0 cm,0) circle (.3cm);
    \draw[thick] (2 cm,0) circle (.3cm);
    \draw[thick] (4 cm,0) circle (.3cm);
    \draw[thick] (0.3 cm,0) -- +(1.4 cm,0) ;
    \draw[thick] (2.3 cm, .1 cm) -- +(1.4 cm,0);
    \draw[thick] (3.1 cm, 0) -- +(-.3cm, -.3cm);
    \draw[thick] (3.1 cm, 0) -- +(-.3cm, .3cm);
    \draw[thick] (2.3 cm, -.1 cm) -- +(1.4 cm,0);
    \node at (0,0)  {$1$};
    \node at (2,0)  {$0$};
    \node at (4,0)  {$1$};
  \end{tikzpicture}
  \end{center}
Here $(\rG_0,\fg_1)\cong (\rGL_1\C\times\rGL_2\C, \C^2\oplus\C^2)$ for the action $(\lambda,A)\cdot (v,w)=(\lambda v A^{-1},Aw^T).$ A point in the open orbit is given by $v=(1,0)$ and $w=(1,0).$ For this point, the $\rG_0$-stabilizer is given by $\lambda=1$ and $A=\smtrx{1&0\\0&\xi}$ for $\xi\in\C^*.$ Geometrically, this parabolic stabilizes an isotropic flag of the form $\C\subset\C^3\subset \C^7$.
\end{example}

\subsection{Jacobson--Morozov regular prehomogeneous vector subspaces}\label{sec JM reg phvss}

Fix a $\Z$-grading $\fg=\bigoplus_{j\in\Z} \fg_j$ with grading element $\zeta$ and consider a prehomogeneous vector space $(\rG_0,\fg_1)$.  Following the proof \cite[Theorem 3.3.1]{CollMcGovNilpotents} of the Jacobson--Morozov theorem, one can show that any nonzero element $e\in\fg_1$ can be completed to an $\fsl_2$-triple $\{f,h,e\}$ with $f\in\fg_{-1}$ and $h\in\fg_0.$

The semisimple element $h$ defines a new grading $\fg=\bigoplus_{j\in\Z} \tilde\fg_j$ where
\[\tilde\fg_j=\{x\in\fg~|~\ad_h(x)=jx\}.\]
Define $\hat\fg_j=\fg_j\cap\tilde\fg_{2j}$ and the subalgebra $\hat\fg\subset\fg$ given by
\begin{equation}
  \label{eq hat algebra}\hat\fg=\bigoplus_{j\in\Z}\hat\fg_j.
\end{equation}

Note that $h$ and $\zeta$ are both in $\hat\fg_0$ and $e\in\hat\fg_1.$ The difference $s=\zeta-\frac{h}{2}$ is semisimple or zero since $\zeta$ and $\frac{h}{2}$ are semisimple and $[\zeta,\frac{h}{2}]=0$. Thus, $\hat\fg$ is reductive since $\hat\fg=\fg^{s}$ is the centralizer of $s.$ 
The following proposition is immediate.
\begin{proposition}\label{prop max jm reg for e}
  Let $\hat\rG_0<\rG$ be the $\rG_0$-centralizer of $h$, then $(\hat\rG_0,\hat\fg_1)$ is a  phvss of $(\rG_0,\fg_1)$ which is JM-regular and $e$ in the open $\hat\rG_0$-orbit $\hat\Omega\subset\hat\fg_1$. 
\end{proposition}

\begin{definition}\label{def max JM reg}
If $e\in\Omega$ is in the open $\rG_0$-orbit, then $(\hat\rG_0,\hat\fg_1)$ will be called a maximal JM-regular phvss of $(\rG_0,\fg_1).$
For any $e\in\fg_1\setminus\{0\}$, we will call $(\hat\rG_0,\hat\fg_1)$ a maximal JM-regular phvss for $e$. 
\end{definition}
\begin{remark}
If $(\rG_0,\fg_1)$ is JM-regular, then for any $e\in\Omega$ an associated $\fsl_2$-triple has the form $\{f,2\zeta,e\}$. In this case, $(\rG_0,\fg_1)=(\hat\rG_0,\hat\fg_1).$
\end{remark}
\begin{example}\label{ex max jm reg}
 Consider the phvs $(\rS(\rGL_{p}\C\times\rGL_q\C), M_{p,q})$ from Example \ref{ex phvs M_p,q} part (2). If $p\neq q$, then the phvs is not JM-regular. For $p>q,$ a maximal JM-regular phvss is isomorphic to $(\rS(\rGL_{q}\C\times\rGL_q\C\times\rGL_{p-q}\C),M_{q,q})$, for the action $(A,B,C)\cdot M=AMB^{-1}.$ 
 In general, the maximal JM-regular phvss of the phvs's $(\rG_0,\fg_1)$ from Example \ref{ex compact dual of herm} are related to the maximal subtube of the associated Hermitian symmetric space.

 For the phvs $(\rGL_p\C\times\rSO_q\C,M_{p,q})$ from Example \ref{ex phvs M_p,q} part (3) with $p>q,$ the phvs is not JM-regular and a maximal JM-regular phvss is isomorphic to $(\rGL_q\C\times \rGL_{p-q}\C\times \rSO_q\C, M_{q,q})$, where the action is given by $(A,B,C)\cdot M=AMC^{-1}$. The maximal JM-regular phvss of $(\rS(\rGL_p\C\times\rGL_q\C\times\rGL_r\C),M_{p,q}\oplus M_{q,r})$ can be described similarly.
 For the regular (non JM-regular) phvs $(\rG_0,\fg_1)=(\rGL_1\C\times \rGL_2\C,\C^2\oplus\C^2)$ from Example \ref{ex reg not jm reg}, the maximal JM-regular phvss $(\hat\rG_0,\hat\fg_1)$ containing $(v,w)=((1,0),(1,0))$ is 
 \[\hat\rG_0=\{(\lambda,A)\in \rGL_1\C\times\rGL_2\C~|~A=\smtrx{a&0\\0&b}\} \ \ \text{and}\ \ \ \hat\fg_1=\langle v\rangle\oplus \langle w\rangle.\]
\end{example}

The construction of a maximal JM-regular phvss of $(\rG_0,\fg_1)$ containing $e\in\fg_1$ depends on $e$ and a choice of $\fsl_2$-triple $\{f,h,e\}$. However, if $e,e'\in\fg_1$ are in the same $\rG_0$-orbit then any two maximal JM-regular phvss's containing $e$ and $e'$ are $\rG_0$-conjugate.

\begin{proposition}\label{prop all max Jm reg G0-conjugate}
  Let $e,e'\in\fg_1$ so that $e'\in\rG_0\cdot e$, let $\{f,h,e\},~\{f',h',e'\}$ be two $\fsl_2$-triples with $h,h'\in\fg_0$. Then there is $g\in\rG_0$ so that 
  \[\{f',h',e'\}=\{\Ad_g f, \Ad_g h,\Ad_g e\}.\]
  In particular, the associated maximal JM-regular phvss's $(\hat\rG_0,\hat\fg_1)$ and $(\hat\rG_0',\hat\fg_1')$ containing $e$ and $e'$ respectively are $\rG_0$-conjugate.
\end{proposition}
\begin{proof}
 We may assume $e=e'.$  
 Let $\fg=\bigoplus_{j\in\Z}\tilde\fg_j$ be the $\Z$-grading associated to the $\fsl_2$-triple $\{f,h,e\}$ and $\fg^e$ be the centralizer of $e$ and $\fu^e=[e,\fg]\cap\fg^e=\fg^e\cap\bigoplus_{j>0}\tilde\fg_j.$ 
 We have $h'-h\in\fu^e$ since $[h-h',e]=0$ and $[e,f-f']=h'-h$. Since $h,h'\in\fg_0$, following the proof \cite[Theorem 3.4.7]{CollMcGovNilpotents}, we can inductively construct $Z\in\fg_0\cap\fu^e$ so that $\Ad_{\exp(Z)}(h)=h+(h'-h)=h'.$ Note that $\Ad_{\exp(Z)}e=e$, and, by the uniqueness part of the Jacobson--Morozov theorem, we have $\Ad_{\exp(Z)}f=f'.$ 
\end{proof}

Finally, we note that every $e\in\fg_1$ defines a parabolic subgroup of $\rP_{0,e}<\rG_0.$ For $e\in\fg_1\setminus\{0\}$, let $\{f,h,e\}$ be an associated $\fsl_2$-triple with $h\in\fg_0.$ If $\fg=\bigoplus_{j\in\Z}\tilde\fg_j$ is the $\Z$-grading with grading element $h,$ then $\tilde\fp=\bigoplus_{j\geq 0}\tilde\fg_j$ is a parabolic subalgebra of $\fg.$ Moreover, $\tilde\fp$ depends only of $e\in\fg_1\setminus\{0\}$ and not the $\fsl_2$-triple $\{f,h,e\}$ (see for example \cite[Remark 3.8.5]{CollMcGovNilpotents}). Note that the parabolic $\tilde\fp$ can also be defined by 
\begin{equation}
  \label{eq: tilde p}\tilde\fp=\{x\in\fg~|~\Ad_{\exp(-t\frac{h}{2})}~\text{is bounded as $t\to\infty$}\}.
\end{equation}
The parabolic subalgebra $\fp_0(e)\subset\fg_0$ is defined by 
\[\fp_{0,e}=\fg_0\cap\tilde\fp.\]
We will denote the associated parabolic subgroup by $\rP_{0,e}<\rG_0.$ Note that $\hat\fg_0\subset\fp_{0,e}$ and $\hat\rG_0<\rP_{0,e}$ define a Levi subalgebra and subgroup respectively.

\begin{proposition}\label{prop: s defines parabolic of g0}
 Let $e\in\fg_1\setminus\{0\}$ and $\{f,h,e\}$ be an $\fsl_2$-triple with $h\in\fg_0.$ If $s=\zeta-\frac{h}{2}$, then  
 \[\fp_{0,e}=\{x\in\fg_0~|~\Ad_{\exp(ts)} x~\text{is bounded as $t\to\infty$}\}\] 
\end{proposition}
\begin{proof}
  The proposition follows immediately from the description of $\tilde\fp$ in \eqref{eq: tilde p} and the fact that $\Ad_{\exp(t\zeta)}x=x$ for all $x\in\fg_0.$
\end{proof}

\section{The Toledo character and period domains}\label{sec toledo char and per dom}

In this section we generalize notions from \cite{BGRmaximalToledo}, which concerned $\Z$-gradings $\fg=\fg_{-1}\oplus\fg_0\oplus\fg_1$, to arbitrary $\Z$-gradings. 
\subsection{The Toledo character and rank}
Let $\rG$ be a complex semisimple Lie group with Lie algebra $\fg$ and Killing form $B.$ Fix a $\Z$-grading $\fg=\bigoplus_{j\in\Z}\fg_j$ with grading element $\zeta.$ Recall that $(\rG_0,\fg_1)$ is a phvs, let $\Omega\subset\fg_1$ be the open $\rG_0$-orbit. Since $\fg_0$ is the centralizer of $\zeta,$ $B(\zeta,-):\fg_0\to\C$ defines a character. 

\begin{definition}\label{def Toledo character}
  The Toledo character $\chi_T:\fg_0\to\C$ is defined by  
  \[\chi_T(x)=B(\zeta,x)B(\gamma,\gamma)~,\]
  where $\gamma$ is the longest root such that $\fg_\gamma\subset\fg_1.$ 
\end{definition}

\begin{remark}
  The normalization factor $B(\gamma,\gamma)$ guarantees that the Toledo character is independent of the choice of invariant bilinear form $B.$ 
\end{remark}

Let $e,e'\in\fg_1\setminus\{0\}$ such that $e'\in\rG_0\cdot e$ and let $\{f,h,e\}$, $\{f',h',e'\}$ be two $\fsl_2$-triples with $h,h'\in\fg_0.$ By Proposition \ref{prop all max Jm reg G0-conjugate}, there is $g\in\rG_0$ so that $\{\Ad_g f,\Ad_g h,\Ad_g e\}=\{f',h',e'\}$. Since $\Ad_g\zeta=\zeta,$ we have
\begin{equation}\label{eq welldefined}
  \chi_T(h)=B(\zeta,h) B(\gamma,\gamma)=B(\Ad_g\zeta,\Ad_g h)B(\gamma,\gamma)=\chi_T(h').
\end{equation}
 As a result, we make the following definition.

\begin{definition}\label{def toledo rank}
  Let $e\in\fg_1$ and $\{f,h,e\}$ be an $\fsl_2$-triple with $h\in\fg_0.$ Define the Toledo rank of $e$ by 
  \[\rk_T(e)=\frac{1}{2}\chi_T(h).\]
  Define the Toledo rank of the phvs $(\rG_0,\fg_1)$ by 
  \[\rk_T(\rG_0,\fg_1)=\rk_T(e) \ \ \text{for $e\in\Omega$.}\]
\end{definition}
\begin{remark}
  Note that if $(\rG_0,\fg_1)$ is a JM-regular phvs, then $\rk_T(\rG_0,\fg_1)=B(\zeta,\zeta)B(\gamma,\gamma).$ 
\end{remark}

\begin{example}
  For the phvs's $(\rG_0,\fg_1)$ from Example \ref{ex compact dual of herm}, the Toledo rank of $(\rG_0,\fg_1)$ agrees with the rank of the associated Hermitian symmetric space. 
\end{example}

\begin{proposition}\label{prop:toledo-rank}
 Let $e\in\fg_1$ and $\{f,h,e\}$ be an associated $\fsl_2$-triples with $h\in\fg_0.$ Then 
 \begin{equation}
 \rk_T(e)=B(\frac{h}{2},\frac{h}{2})B(\gamma,\gamma).\label{eq:3}
\end{equation}
\end{proposition}
\begin{proof}
  Write $\zeta=\frac{h}{2}+s$. Note that $s\in\fg^e$ and $\frac{h}{2}\in\ad_e(\fg).$ Thus, $B(\frac{h}{2},s)=0$ and
  \[\rk_T(e)=B(\zeta,\frac{h}{2})B(\gamma,\gamma)=B(\frac{h}{2}+s,\frac{h}{2})B(\gamma,\gamma)=B(\frac{h}{2},\frac{h}{2})B(\gamma,\gamma).\]
\end{proof}

The following proposition will be used often in subsequent sections.
\begin{proposition}\label{prop JM reg toledo has rel invar}
  If $(\rG_0,\fg_1)$ is a JM-regular phvs, then there is positive integer multiple $q\cdot \chi_T$ of the Toledo character which lifts to a character $\chi_{T,q}:\rG_0\to\C^*$ which has a relative invariant of degree $q\cdot\rk_T(\rG_0,\fg_1).$
\end{proposition}

\newcommand{\revdots}{\mathinner{\mkern1mu\raise1pt\vbox{\kern7pt\hbox{.}}\mkern2mu\raise4pt\hbox{.}\mkern2mu\raise7pt\hbox{.}\mkern1mu}}
\begin{example}\label{calc_SO2pq}
  We illustrate these notions in the case of Example \ref{running_example}, with $G=\rSO_{2p+q}\C$ and $G_0=\rGL_p\C\times \rSO_q\C$. Taking an isotropic basis of $\C^{2p+q}=\C^p\oplus\C^q\oplus\C^p$, such that $\langle e_i,e_{2p+q+1-i}\rangle=1$, we have the grading element
  \[ \zeta =
    \begin{pmatrix}
      \Id_p &  &  \\ & 0 & \\ & & - \Id_p
    \end{pmatrix}.\]
  Then $\fg_1=\Hom(\C^p,\C^q)$, where $u\in \Hom(\C^p,\C^q)$ represents the matrix
\[ e =
  \left(\begin{array}{c|c|c}
    \phantom{u} & - u^T & \\ \hline & & u \\ \hline & &
  \end{array}\right), \text{ where } u^T_{ij}=u_{p+1-j,q+1-i}.\]
The orbits are classified by two integers $(r_1,r_2)$, such that the image of $u\in \Hom(\C^p,\C^q)$ is the sum of a nondegenerate subspace of dimension $r_1$ and a totally isotropic subspace of dimension $r_2$. Let $J_r$ be the rank $r$ matrix \[ J_r = \left(   \begin{smallmatrix}
    & & 1 \\ & \revdots & \\ 1 & &
  \end{smallmatrix}\right),\]
then one can take
\[ u = \begin{pmatrix}
    \Id_{r_2} & & \\ & \Id_{r_1}  & \\ & 0 & \\ & J_{r_1} &  \\   & & 0
  \end{pmatrix}, \qquad
  h = \left(\begin{array}{ccc|ccc|ccc}
    0 & & & & & & & & \\
    & 2\Id_{r_1} & & & & & & & \\
    & & \Id_{r_2} & & & & & & \\ \hline
    & & & \Id_{r_2} & & & & & \\
    & & & & 0 & & & & \\
    & & & & & -\Id_{r_2} & & & \\ \hline
    & & & & & & -\Id_{r_2} & & \\
    & & & & & & & -2 \Id_{r_1} & \\
    & & & & & & & & 0
  \end{array}\right).
\]
(The formula for $u$ is valid if $r_2+r_1\leq \frac q2$, the reader will modify $u$ accordingly if $r_2+r_1>q$; the formula for $h$ remains the same).
Take the standard invariant form $B(X,Y)=\tr(XY)$.
If $q>1$ we have $B(\gamma,\gamma)=1$ and therefore $\rk_Te=\frac14 B(h,h)=2r_1+r_2$. On the other hand $\chi_T(x)=B(\zeta,x)$, so that we will have $\tau(E,\varphi)=2\deg V$ if the Higgs bundle $(E,\varphi)$ is given as $V\oplus W\oplus V^*$ with $V$ a $\rGL_p\C$-bundle and $W$ a $\rSO_q\C$-bundle. This justifies the inequality (\ref{eq:4}) as a consequence of Corollary \ref{cor am ineq for comp}. 

\end{example}

\subsection{Real forms and period domains}\label{sec real forms}
Let $\fg$ be a complex semisimple Lie algebra and consider and conjugate linear involution $\sigma:\fg\to\fg.$ The fixed point set $\fg^\sigma\subset\fg$ is a real subalgebra such that $\fg^\sigma\otimes\C=\fg.$ Such a subalgebra is called a real form of $\fg.$ On the level of groups, a real form $\rG^\sigma<\rG$ is the fixed point set of an anti-holomorphic involution $\sigma:\rG\to\rG.$ A real form is called compact if $\rG^\sigma$ is compact or $\fg^\sigma$ is Lie algebra of a maximal compact subgroup. Compact real forms exist and are unique up to conjugation. 

Real forms of $\fg$ can be equivalently defined in terms of complex linear involutions $\theta:\fg\to\fg.$ Namely, Cartan proved that, given a real form $\sigma$, there is a compact real form $\tau:\fg\to\fg$ such that $\sigma\circ\tau=\tau\circ \sigma$, and that given a complex linear involution $\theta$, there is a compact real form $\tau:\fg\to\fg$ so that $\theta\circ \tau=\tau\circ \theta.$ The correspondence is then given by setting $\theta=\sigma\circ\tau$. 
A real subalgebra $\sigma:\fg\to\fg$ is said to be of {\bf Hodge type} if $\sigma$ is an inner automorphism of $\fg$. If $\tau$ is a compact real form so that $\theta=\tau\circ\sigma=\sigma\circ\tau,$ then $\sigma$ is of Hodge type if and only if there is a Cartan subalgebra $\ft\subset\fg$ such that $\theta|_{\ft}=\Id.$ 

Given a complex linear involution $\theta:\fg\to\fg$, we will write $\fg=\fh\oplus\fm$ for $\pm1$-eigenspaces of $\theta$, namely $\theta|_\fh=\Id$ and $\theta_\fm=-\Id$. We will call the decomposition $\fg=\fh\oplus\fm$ the {\bf complexified Cartan decomposition} of a real form $\fg^\R\subset\fg$.
There is a real form of Hodge type associated to every $\Z$-grading of $\fg$.
\begin{proposition}\label{prop hodge real form of Zgrad}
  Let $\fg=\bigoplus_{j\in\Z}\fg_j$ be a $\Z$-grading, and define $\theta:\fg\to\fg$ by 
  \[\theta|_{\fg_j}=(-1)^j\Id.\]
  Then $\theta$ is a Lie algebra involution which defines a real form of Hodge type. 
\end{proposition}
\begin{proof}
  It is clear that $\theta$ is a Lie algebra involution. Moreover, from the constructions in \S\ref{section: canonical Z grade}, $\fg_0$ contains a Cartan subalgebra of $\fg.$ Hence the real form associated to the involution $\theta$ is of Hodge type.
\end{proof}
\begin{example}\label{ex real forms of examples}
  For gradings of the form $\fg=\fg_{-1}\oplus\fg_0\oplus\fg_1$, the canonical real forms are exactly the set of real forms of Hermitian type. For example, the real form of $\rSL_{p+q}\C$ associated to the phvs $(\rS(\rGL_p\C\times\rGL_q\C, M_{p,q})$ from Examples \ref{ex phvs M_p,q} and \ref{ex parabolic type} is $\rSU(p,q).$ 

The real form of $\rSO_{2p+q}\C$ associated to the phvs $(\rGL_p\C\times\rSO_q\C,M_{p,q})$ from Examples \ref{ex phvs M_p,q} and \ref{ex parabolic type} is $\rSO(2p,q).$ Similarly, the real form of $\rSL_{p+q+r}\C$ associated to the phvs $(\rS(\rGL_p\C\times\rGL_q\C\times\rGL_r\C),M_{p,q}\oplus M_{q,r})$ is $\rSU(p+r,q).$
For the regular (non JM-regular) phvs from Example \ref{ex reg not jm reg}, the associated real form of $\rSO_7\C$ is $\rSO(3,4).$ For the JM-regular phvs associated to a principal $\fsl_2$-triple (Example \ref{ex principal}), the associated real form is the split real form if the split real form is of Hodge type, otherwise it is the quasisplit real form which is nonsplit. For example, for $\rG=\rSL_{2p}\C,$ the split real form is $\rSL_{2p}\R$ but the canonical real form associated to the grading is $\rSU(p,p).$
\end{example}
\begin{remark}\label{rem compact inv and Z-grading}
  Given a compact real form $\tau,$ $B(-\tau(x),y)$ defines a  nondegenerate positive definite Hermitian inner-product on $\fg.$ Thus, $\tau(\fg_j)=\fg_{-j}$ for any $\Z$-grading $\fg=\bigoplus_j\fg_j$.
\end{remark}

Let $\theta$ be the involution from Proposition \ref{prop hodge real form of Zgrad} and $\tau$ be a compact real form such that $\sigma=\theta\circ\tau:\fg\to\fg$ is the associated real form of Hodge type. 
Let $\rG^\R<\rG$ be the associated real form of $\rG$ and let $\rH^\R<\rG^\R$ be the associated maximal compact subgroup. 
Note that the real form $\sigma$ restricts to a compact real form on $\fg_0,$ set $\rH_0^\R=\rG_0\cap\rG^\R.$ 

Consider the homogeneous space 
\[D=\rG^\R/\rH_0^\R.\]
Note that there is a fibration $D\to\rG^\R/\rH^\R$ over the Riemannian symmetric space of $\rG^\R.$ In fact, the homogeneous space $D$ has a natural  homogeneous complex structure. Indeed, the tangent bundle of $D=\rG^\R/\rH_0^\R$ is isomorphic to the associated bundle 
 \[TD\cong \rG^\R\times_{\rH_0^\R}\fg^\R/\fh_0^\R~.\]
Let $\fg^\R=\fh_0^\R\oplus\fq^\R$ be an orthogonal decomposition and $\fq=\fq^\R\otimes \C$. Then the complexified tangent bundle is $T_\C D\cong\rG^\R\times_{\rH_0^\R}\fq,$ and a complex structure on $D$ 
 is equivalent to an $\rH_0^\R$-invariant decomposition $\fq=\fq_-\oplus\fq_+$ such that $\sigma(\fq_+)=\fq_-$ and $[\fq_\pm,\fq_\pm]\subset\fq_\pm$. The decomposition $\fq=\fq_-\oplus\fq_+$ is given by setting 
 \[\xymatrix{\fq_+=\bigoplus\limits_{j>0}\fg_j&\text{and}&\fq_-=\bigoplus\limits_{j<0}\fg_j}.\]
The complex manifold $D$ is called a {\bf period domain} for $\rG^\R$.
Note that the holomorphic tangent bundle of a period domain decomposes as 
\[T^{1,0}_\C D\cong \rG^\R\times \rH_0^\R\fq_+=\bigoplus_{j>0} \rG^\R\times_{\rH_0^\R}\fg_j.\]
\begin{example}
  For the phvs's from Example \ref{ex compact dual of herm}, $\rG^\R<\rG$ is a Hermitian real form and $\rH_0^\R<\rG^\R$ is a maximal compact subgroup. Hence the period domain $D=\rG^\R/\rH_0^\R$ is the Riemannian symmetric space for $\rG^\R.$
 For the phvs $(\rG_0,\fg_1)$ from Example \ref{ex principal}, $\rG_0<\rG$ is a Cartan subgroup, $\rG^\R$ is the quasisplit real form of Hodge type and the period domain is $D=\rG^\R/\rU(1)^{\rk(\fg)}$.
  For the phvs $(\rGL_{p}\C\times\rSO(p,\C), M_{p,q})$ from Examples \ref{ex phvs M_p,q} and \ref{ex real forms of examples}, $\rG^\R=\rSO(2p,q)$ and the period domain is $D=\rSO(2p,q)/(\rU_{p}\times \rSO_q).$ 
\end{example}
\subsection{The Toledo character and holomorphic sectional curvature}\label{sec hol sec curv} 

Here we discuss how the choice of the Toledo character is related to a metric of minimal holomorphic sectional curvature $-1$ on $\rG^\R/\rH_0^\R$.

As is well-known, one can define an invariant Hermitian metric on $\rG^\R/\rH_0^\R$ by using the Hermitian scalar product
\[ \langle x,y \rangle = B(x^*,y) \quad \text{ if } x,y\in \fg_k, \]
where $x^*=-\tau(x)$ and $\tau$ is the compact conjugation fixed on $\fg$. This does not give a K\"ahler metric (the corresponding 2-form is not closed) but one can still define the holomorphic sectional curvature $K$ of its Chern connection. We have the well-known formula (see for example \cite{PeriodDomainsCMP}):
\[
K(x)=- \frac{|[x,x^\ast]|^2}{|x|^4}.
\]
We are interested in the function $K$ on $\fg_1$. It is known (and reproved quickly below) that critical points of $K$ are obtained on elements $e$ such that $\{e^*,h=[e,e^*],e\}$ is a $\fsl_2$-triple. For such $e$, we have $|[e,e^*]|^2=\langle [[e,e^*],e],e \rangle=2|e|^2$, and $|e|^2=\langle [\zeta,e],e \rangle=\langle\zeta,[e,e^*]\rangle$, therefore
\[ K(e) = - \frac 2 {B(\zeta,h)}, \]
which is, up to a constant, the inverse of the Toledo rank defined earlier. Then we will prove the following.

\begin{proposition}\label{prop:sectional_curvature}
  The maximum of $K$ on $\fg_1$ is attained on the regular orbit, and is equal to $-\frac2{B(\zeta,h_{\rm reg})}$, where $e$ is a regular element such that $\{e^*,h_{\rm reg}=[e,e^*],e\}$ is a $\fsl_2$-triple.

  The minimum of $K$ on $\fg_1$ is attained on a minimal orbit, and is equal to $-B(\gamma,\gamma)$ for a long root $\gamma\in\Delta_1$ (the minimum is attained on $e_\gamma$, which belongs to a minimal orbit).

  Therefore, after multiplying the Hermitian metric by $B(\gamma,\gamma)$, we obtain a normalized metric with normalized holomorphic sectional curvature
    \[ -1 \leq K_{\rm norm} \leq - \frac1{\rk_T(G_0,\fg_1)}. \]
    At a general critical point $e\in\fg_1$ of the curvature (therefore  $\{e^*,[e,e^*],e\}$ is a $\fsl_2$-triple), we have the normalized value $K_{\rm norm}(e) = - \frac1{\rk_T(e)}$.
\end{proposition}

From the Toledo character we obtain a $\rG^\R$-invariant 2-form on $\rG^\R/\rH_0^\R$ by
\[
\omega(x,y)=i\chi_T([x,y]),\quad \text{ for } x,y\in \fq^\R.
\]
This actually defines a pseudo-K\"ahler metric, which coincides with the previous normalized metric in the horizontal directions (that is in $\fg_1$). This explains the choice of the normalization for the Toledo character.

\begin{proof}[Proof of Proposition \ref{prop:sectional_curvature}]
  For completeness, we begin by proving that the critical points come from $\fsl_2$-triples $\{e^*,[e,e^*],e\}$. Since $K(e)$ is invariant by homothety, we can restrict to variations $\dot e\perp e$. Then
  \[ \dot K = - \frac2{|e|^4} \langle [e,e^*],[\dot e,e^*]+[e,\dot e^*]\rangle
    =  - \frac4{|e|^4} \Re \langle [[e,e^*],e], \dot e\rangle \]
It follows that for a critical point, $[[e,e^*],e]=\lambda e$ and up to renormalizing $e$ one can suppose $\lambda=2$.

We will now establish a second variation formula. Start from a $\fsl_2$-triple $(h=[e,e^*],e,e^*)$ and take a vector $x\perp e$ such that $|x|=|e|$. Consider
\[ e(t) = \cos(t) e + \sin(t) x = (1-\frac{t^2}2) e + t x + o(t^2). \]
Then $e(t)$ has constant norm, and, up to order 2,
\[ [e(t),e(t)^*] = [e,e^*] + t ( [x,e^*]+[e,x^*]) + t^2 ([x,x^*]-[e,e^*]) . \]
Therefore, again up to order 2,
\begin{align*}
  - |e|^4 K(e(t)) &= |[e(t),e(t)^*]|^2 \\
  &= |[e,e^*]|^2 + t^2 \big( |[x,e^*]+[e,x^*]|^2 + 2 \Re \langle [e,e^*],[x,x^*]-[e,e^*] \rangle \big).
\end{align*}
We have already noticed above that $|[e,e^*]|^2=2|e|^2=2|x|^2$, and therefore
\begin{align}
  \left.\frac{d^2}{dt^2}\right|_{t=0}|e|^4 K(e(t)) &= 2\left( 4 |x|^2 - 2 \langle [h,x],x \rangle -  |[x,e^*]+[e,x^*]|^2\right) \notag \\
  &= 4 \langle (2-\ad h)x,x \rangle -  2 |[x,e^*]+[e,x^*]|^2. \label{eq:2}
\end{align}

We first deduce that the maximum is attained on the regular orbit. Observe that $e$ is regular, that is the orbit of $e$ is open, if and only if $\fg_1 \cap \ker \ad e^*=0$. Therefore, if $e$ is not regular, we can take $x\in \ker \ad e^*$. Since $\ad h \leq 0$ on $\ker \ad e^*$, it follows from (\ref{eq:2}) that in the direction of $x$ we have
    \[ \left.\frac{d^2}{dt^2}\right|_{t=0}|e|^4 K(e(t)) \geq 8 |x|^2 \]
and therefore the maximum can not be attained at $e$. This proves the claim that the maximum is obtained on the regular orbit.

To find the minimum, we use the following fact \cite[Corollary 7, p. 100]{ManievelPrehom}: denote $\fu$ the subspace of $\fg_1$ given as the sum of the eigenspaces of $\ad h$ for the eigenvalues at least $2$. Then all orbits in the closure of the orbit of $e$ meet $\fu$. Therefore suppose that the orbit of $e$ is not minimal. It follows that we can take $x\in \fu$, and (\ref{eq:2}) gives us, in the direction of $x$,
    \[ \left.\frac{d^2}{dt^2}\right|_{t=0}|e|^4 K(e(t)) \leq  - 2 |[x,e^*]+[e,x^*]|^2. \]
    Such $x$ can not be in the kernel of $\ad e^*$ ($\ad h\leq 0$ on this kernel). Therefore the right hand side is nonzero, so that the minimum cannot be attained at $e$. Therefore the minimum can be attained only at a minimal orbit.

    There can be several minimal orbits if the prehomogeneous space is not irreducible. They are the orbits of elements $e_\gamma$ for $\gamma\in\Delta$ the longest root. As we have seen, $K(e_\gamma)=-\frac2{|e_\gamma|^2}=-\frac4{|h_\gamma|^2}=- B(\gamma,\gamma)$. The result follows.
\end{proof}

\begin{remark}\label{rem:holo-curv}
  If we fix an orbit $\cO$, then $K_{\rm norm}$ must have a maximum on $\overline{\cO}$. It follows from the proof of the proposition, and in particular of the description of the orbits in the closure of $\cO$, that the maximum on $\overline{\cO}$ is achieved at some element $e\in\cO$ such that $(e^*,[e,e^*],e)$ is a $\fsl_2$-triple; it is therefore equal to $K_{\rm norm}(e)=-\frac1{\rk_T(e)}$. In particular, we obtain that for any $x\in\overline{\cO}$ we have
  \begin{equation}
    \label{eq:1}
    -1 \leq K_{\rm norm}(x) \leq -\frac1{\rk_T(e)}.
  \end{equation}
\end{remark}

\begin{example}
  The case $\rG=\rSL_n\C$ is particularly simple to calculate. A first observation is that any nilpotent element $e\in\fsl_n\C$ belongs to the regular orbit of the nilpotent subalgebra of a parabolic algebra of $\fsl_n\C$. It follows that $e$ can be considered as a regular element of a $\fg_1$ for some grading of $\fsl_n\C$. This observation is not essential for our calculation but relates every nilpotent $e$ to variations of Hodge structures.

  We now calculate the corresponding bounds for the holomorphic sectional curvature: \[ -1\leq K\leq -\frac1{\rk_Te}, \] where, by (\ref{eq:3}), we have $\rk_Te=\frac14 B(h,h)B(\gamma,\gamma)=\frac12 B(h,h)$ since $B(\gamma,\gamma)=2$ in $\fsl_n$  for the standard invariant form $\tr(XY)$. If we have a Jordan block of size $k$, the eigenvalues of $h$ are $k-1$, $k-3$, ..., $-(k-1)$, therefore $B(h,h)=(k-1)^2+(k-3)^2+\cdots+(-(k-1))^2=\frac13 k(k^2-1)$. Finally, if $e$ has Jordan blocks of size $k_1$, $k_2$, ..., $k_j$ we obtain
  \[\rk_Te = \frac16 \sum_1^j k_i(k_i^2-1). \]
  This gives immediately  some curvature bounds in \cite{Qiongling_Li}.
\end{example}

\begin{proposition}
 Let $e\in\fg_1$ and let $\Omega\subset \fg_1$ be the open $\rG_0$-orbit of 
the phvs 
$(\rG_0,\fg_1)$. Then  
$$
0\leq \rk_T(e)\leq\rk_T(\rG_0,\fg_1),
$$  
with equality if and only if $e\in\Omega$.
\end{proposition}
\begin{proof}
We give an indirect proof applying  Proposition \ref{prop:sectional_curvature} and Remark \ref{rem:holo-curv}. The Toledo rank is related to the normalized holomorphic sectional curvature $K_{\rm norm}$ of the corresponding period domain by $K_{\rm norm}(e)=-\frac1{\rk_T(e)}$ if $(f,h,e)$ is a \emph{real} $\fsl_2$-triple 
(in a given orbit one can always find a $e$ which is part of such a triple, see 
 Remark \ref{rem:holo-curv}). Then it is proved that $K_{\rm norm}(e)\leq K_{\rm norm}(e')$ if the orbit of $e$ is included in the closure of the orbit of $e'$. Moreover the maximum of $K_{\rm norm}$ is attained only on the regular orbit. The proposition follows.  
\end{proof}

\section{Higgs bundles and variations of Hodge structure}
\label{sec: higgs bundles nah}
For this this section, let $X$ be a compact Riemann surface of genus $g\geq 2$ and let $K$ be its canonical bundle.  

\subsection{Higgs bundles and Hodge bundles}
Let $\rG$ be a complex reductive Lie group with Lie algebra $\fg$ and nondegenerate $\rG$-invariant bilinear $\langle\cdot,\cdot\rangle$. Let $\rho:\rG\to\rGL(V)$ be a holomorphic representation. 
If $E\to X$ is a $\rG$-bundle, we will denote the $V$-bundle $E\times_\rG V$ associated to $E$ via the representation $\rho$ by $E(V).$

\begin{definition}
A $(\rG,V)$-Higgs pair is a pair $(E,\varphi)$ where $E$ is a holomorphic principal $\rG$-bundle on $X$ and $\varphi$ is a holomorphic section of $E(V)\otimes K.$ 
\end{definition}

When $V=\fg$ and the representation $\rho$ is the adjoint representation, a $(\rG,\fg)$-Higgs pair is called a $\rG$-Higgs bundle. 
Suppose $\rG^\R<\rG$ is a real form with complexified maximal compact $\rH<\rG$ and complexified Cartan decomposition $\fg=\fh\oplus\fm$. When $\rho:\rH\to\rGL(\fm)$ is the restriction of the adjoint representation of $\rG$, then an $(\rH,\fm)$-Higgs pair is called a $\rG^\R$-Higgs bundle. 

\begin{remark}
When $\rG^\R$ is compact, a $\rG^\R$-Higgs bundle is just a holomorphic $\rG$-bundle. When $\rG^\R=\rG$ (viewed as a real form of $\rG\times\rG$), a $\rG^\R$-Higgs bundle is a $\rG$-Higgs bundle. 
\end{remark}
If $E$ is a principal $\rG$-bundle and $\hat\rG<\rG$ is a subgroup, then a structure group reduction of $E$ to $\hat\rG$ is a section $\sigma$ of the bundle $E(\rG/\hat\rG).$ Associated to such a reduction is a principal $\hat\rG$-subbundle $E_\sigma\subset E$ such that $E_\sigma(\rG)$ is canonically isomorphic to $E.$
\begin{definition}\label{def higgs pair reduction}
  Let $\rG$ be a complex reductive Lie group and $\rho:\rG\to\rGL(V)$ be a holomorphic representation. Let $\hat\rG<\rG$ and $\hat V\subset V$ be a $\rho(\hat\rG)$-invariant subspace. We say that a $(\rG,V)$-Higgs pair $(E,\varphi)$ reduces to a $(\hat\rG,\hat V)$-Higgs pair, if there is a holomorphic reduction $E_{\hat\rG}$ of $E$ to $\hat\rG$ such that $\varphi\in H^0(E_{\hat\rG}(\hat V)\otimes K)\subset H^0(E_{\hat\rG}(V)\otimes K).$
\end{definition}
For example, if $\rG^\R<\rG$ is a real form, $\rH<\rG$ is the complexification of a maximal compact of $\rG^\R$ and $\fg=\fh\oplus\fm$ is complexified Cartan decomposition, then a $\rG$-Higgs bundle $(E,\varphi)$ reduces to a $\rG^\R$-Higgs bundle if there is a holomorphic reduction $E_{\rH}$ of $E$ to $\rH$ such that $\varphi\in H^0(E_{\rH}(\fm)\otimes K).$

\begin{example}\label{ex sl2 uniformizing}
  Let $\{f,h,e\}$ be a basis for $\fsl_2\C$, let $\rT<\rP\rSL_2\C$ be the subgroup with Lie algebra $\langle h\rangle.$ Note that $\rT\cong\C^*$ and the adjoint action of $\rT$ on $\langle e\rangle$ is given by $\lambda\cdot e=\lambda e.$ So, if $E_\rT$ is the frame bundle of $K^{-1},$ then the associated bundle $E_\rT(\langle e\rangle)\otimes K\cong\cO_X.$ Hence we have an $\rP\rSL_2\C$-Higgs bundle 
  \[(E_\rT(\rP\rSL_2\C),e).\]
Since $\deg(K)$ is even, this example can be lifted to $\rSL_2\C.$ The lifted action of the subgroup $\C^*\cong\hat\rT<\rSL_2\C$ with Lie algebra $\langle h\rangle$ is given by $\lambda\cdot e=\lambda^2e.$ As a result, if $E_{\hat \rT}$ is the frame bundle of a square root $K^{-\frac{1}{2}}$ of $K^{-1}$, then $(E_{\hat\rT}(\rSL_2\C),e)$ defines an $\rSL_2\C$-Higgs bundle. These Higgs bundles will be the fundamental building blocks for the results of \S\ref{sec rigidity}.
\end{example}

Now suppose $\rG$ is a complex reductive Lie group and $\fg=\bigoplus_{j\in\Z}\fg_j$ is a $\Z$-grading of its Lie algebra with grading element $\zeta\in\fg_0.$ Let $\rG_0<\rG$ be the centralizer of $\zeta.$ 
Note that $\exp(\lambda\zeta)$ is in the center of $\rG_0$ and $\Ad(\exp(\lambda\zeta))$ acts on each $\fg_j$ by $\lambda^j\cdot\Id.$
Let $(E_{\rG_0},\varphi)$ be a $(\rG_0,\fg_k)$-Higgs pair. Note that the central element $\exp(\frac{\lambda}{k}\zeta)\in\rG_0$ defines a holomorphic automorphism of $E_{\rG_0}$ and acts on $\varphi$ by multiplication by $\lambda$. As a result we have an isomorphism of $(\rG_0,\fg_k)$-Higgs pairs
 \begin{equation}
   \label{eq iso of G0gk pairs}(E_{\rG_0},\varphi)\cong(E_{\rG_0},\lambda\varphi) \text{~ \ for all $\lambda\in\C^*$}.
 \end{equation}
Extending the structure group defines a $\rG$-Higgs bundle $(E_\rG,\varphi)$ 
\[(E_\rG,\varphi)=(E_{\rG_0}(\rG),\varphi)\]
since $E_{\rG_0}(\fg_k)\subset E_{\rG_0}[\fg]\cong E_{\rG}(\fg)$. Moreover, $(E_\rG,\varphi)\cong (E_{\rG},\lambda\varphi)$ for all $\lambda\in\C^*.$ 

\begin{definition}\label{def hodge bundle}
  A $\rG$-Higgs bundle $(E,\varphi)$ is called a {\bf Hodge bundle of type $(\rG_0,\fg_k)$} if it reduces to a $(\rG_0,\fg_k)$-Higgs pair.
\end{definition} 
\begin{example}
 The Higgs bundles from Example \ref{ex sl2 uniformizing} are Hodge bundles for the grading $\fg=\fg_{-1}\oplus\fg_0\oplus\fg_1=\langle f\rangle \oplus\langle h\rangle\oplus \langle e\rangle.$ In these cases, $\rG_0$ is a Cartan subgroup of $\rG.$ 
\end{example}
Given a $\Z$-grading $\fg=\bigoplus_{j\in\Z}\fg_j$, we can define a subalgebra $\tilde\fg$ consisting of summands $\fg_j$ with $j=0\mod k.$ Note that $\tilde\fg$ has an associated $\Z$-grading with $\fg_k=\tilde\fg_1.$ Let $\tilde\rG<\rG$ be the associated reductive subgroup. The following proposition is immediate.
\begin{proposition}\label{prop G0g1 type reduction}
  Let $(E,\varphi)$ be a $\rG$-Higgs bundle which is a Hodge bundle of type $(\rG_0,\fg_k)$. Then $(E,\varphi)$ reduces to a $\tilde\rG$-Higgs bundle and, as a $\tilde\rG$-Higgs bundle, it is a Hodge bundle of type $(\rG_0,\tilde\fg_1).$
\end{proposition}

As a result, we will usually consider Hodge bundles of type $(\rG_0,\fg_1)$. Recall from Proposition \ref{prop hodge real form of Zgrad}, that there is a canonical real form $\rG^\R<\rG$ of Hodge type associated to a $\Z$-grading. The complexified Cartan decomposition $\fg=\fh\oplus\fm$ satisfies $\fg_{2j}\subset\fh$ and $\fg_{2j+1}\subset\fm,$ in particular $\rG_0<\rH.$ 
As a result, a $\rG$-Higgs bundle which is a Hodge bundle of type $(\rG_0,\fg_1)$ reduces to a $\rG^\R$-Higgs bundle. Combining this observation with Proposition \ref{prop G0g1 type reduction} gives the following result which was first observed by Simpson \cite[\S4]{localsystems}.

\begin{proposition}
Let $(E,\varphi)$ be a $\rG$-Higgs bundle which is a Hodge bundle of type $(\rG_0,\fg_k)$. Then $(E,\varphi)$-reduces to a $\tilde\rG^\R$-Higgs bundle, where $\tilde\rG^\R<\rG$ is a real group of Hodge type. In fact, $\tilde\rG^\R$ is the real form of the subgroup $\tilde\rG<\rG$ from Proposition \ref{prop G0g1 type reduction} associated to the grading of $\tilde\fg.$
\end{proposition}
 \begin{example}
   The Hodge bundles in Example \ref{ex sl2 uniformizing}, thus reduce to the real forms $\rPSL_2\R<\rPSL_2\C$ and $\rSL_2\R<\rSL_2\C.$
 \end{example}

The above results also apply to $\rG^\R$-Higgs bundles. Namely, fix a real form $\rG^\R<\rG$ with a maximal compact subgroup $\rH^\R<\rG^\R$. Let $\rH<\rG$ be the complexification of $\rH^\R$ and $\fg=\fh\oplus\fm$ be a complexified Cartan decomposition. 
Consider a $\Z$-grading of $\fg$ given by $\fg=\bigoplus_{j\in\Z}\fh_j\oplus\fm_j$ with grading element $\zeta\in\fh_0.$ Let $\rH_0<\rH$ be the centralizer of $\zeta.$ 
Given an $(\rH_0,\fm_k)$-Higgs pair $(E_{\rH_0},\varphi)$, extending the structure group to $\rH$ defines a $\rG^\R$-Higgs bundle $(E_{\rH_0}(\rH),\varphi)$ such that $(E_{\rH_0}(\rH),\varphi)\cong (E_{\rH_0}(\rH),\lambda\varphi)$ for all $\lambda\in\C^*$

A $\rG^\R$-Higgs bundle $(E,\varphi)$ is called a {\bf Hodge bundle of type $(\rH_0,\fm_k)$} if it reduces to a $(\rH_0,\fm_k)$-Higgs pair. Consider the subalgebra $\tilde\fg=\tilde\fh\oplus\tilde\fm$ given by 
\[\xymatrix{\tilde\fh=\bigoplus\limits_{j=0\mod 2k}\fh_j&\text{and}&\tilde\fm=\bigoplus\limits_{j=k\mod 2k}\fm_j}.\]
Let $\tilde\fg^\R\subset\tilde\fg$ be the associated real form of Hodge type with complexified Cartan decomposition $\tilde\fg=\tilde\fh\oplus\tilde\fm.$ As in the complex case, we have the following result. 
\begin{proposition}
  Let $(E,\varphi)$ be a $\rG^\R$-Higgs bundle which is a Hodge bundle of type $(\rH_0,\fm_k).$ Then there is a reductive subgroup $\tilde\rG^\R<\rG^\R$ of Hodge type such that $(E,\varphi)$ reduces to a $\tilde\rG^\R$-Higgs bundle. Moreover, the resulting $\tilde\rG^\R$-Higgs bundle is a $(\rH_0,\tilde\fm_1)=(\tilde\rG_0,\tilde\fg_1)$-Hodge bundle.
\end{proposition}
As a result, we will mostly consider Hodge bundles of type $(\rG_0,\fg_1)$.

\subsection{Moduli spaces and fixed points}
\label{sec moduli fixed}
To form a moduli space of Higgs pairs, we need to define suitable notions of stability. We describe this below and refer to \cite{HiggsPairsSTABILITY,RelKHpairscorresp} for more details. Let $\rG$ be a complex reductive Lie group and $\rho:\rG\to\rGL(V)$ be a holomorphic representation. Fix a maximal compact subgroup $\rK^\R<\rG$ and let $\fk^\R$ be its Lie algebra. An element $s\in i\fk^\R$ defines subspaces of $V$ via the representation $\rho$
\[\xymatrix@=1em{V_s^0=\{v\in V~|\rho(e^{ts})v= v \}&\text{and}&V_s=\{v\in V~|~\rho(e^{ts})(v) \text{ is bounded as $t\to\infty$}\}}.\]
When $\rho:\rG\to\rGL(\fg)$ is the adjoint representation, $\fg_s=\fp_s\subset\fg$ is a parabolic subalgebra with Levi subalgebra $\fg_s^0=\fl_s\subset\fp_s.$ The associated subgroups $\rL_s<\rP_s$ are given by
\[\xymatrix@=1em{\rL_s=\{g\in\rG~|~\Ad(g)s=s\}&\text{and}&\rP_s=\{g\in\rG~|~\Ad(e^{ts})(g)\ \text{is bounded as $t\to\infty$}\}}.\]
Moreover, $s$ defines a character $\chi_s:\fp_s\to\C$ by
\[\chi_s(x)=\langle s,x\rangle \text{ for $x\in\fp_s$.}\]

Let $E$ be a $\rG$-bundle, $s\in i\fk^\R$ and $\rP_s<\rG$ be the associated parabolic subgroup. A reduction of structure group of $E$ to $\rP_s$ is a $\rP_s$-subbundle $E_{\rP_s}\subset E$, this is equivalent to a section $\sigma\in\Gamma(E(\rG/\rP_s))$ of the associated bundle. We will denote the associated $\rP_s$-subbundle by $E_\sigma.$ The degree of such a reduction will be defined using Chern-Weil theory. Since $\rP_s$ is homotopy equivalent to the maximal compact $\rK_s^\R=\rK^\R\cap\rL_s$ of $\rL_s$, given a reduction of structure $\sigma$ of $E$ to $\rP_s$, there is a further reduction of $E$ to $\rK_s^\R$ which is unique up to homotopy. Let $E_{\sigma'}\subset E$ be the resulting $\rK^\R_s$ principal bundle. The curvature $F_A$ of a connection $A$ on $E_{\sigma'}$ satisfies $F_A\in\Omega^2(X,E_{\sigma'}(\fk_s^\R)).$ Thus, evaluating the character $\chi_s$ on the curvature  we have $\chi_s(F_A)\in\Omega(X,i\R),$ and we define the degree of $\sigma$ as
\begin{equation}
   \label{eq deg of parabolic reduction}\deg E(\sigma,s)=\frac{i}{2\pi}\int_X\chi_s(F_A).
 \end{equation} 

 If a multiple of $q\cdot\chi_s$ lifts to a character $\tilde\chi_s:\rP_s\to\C^*$ and $\sigma\in\Gamma(E(\rG/\rP_s))$ is a reduction, then $E_{\sigma}\times_{\chi_s}\C^*=E_{\sigma}(\tilde\chi_s)$ is a line bundle and $\deg E(\sigma,s)=\frac{1}{q}\deg E(\tilde\chi_s).$
When $s$ is in the center $\fz$ of $\fg$ then $\rP_s=\rG$. In this case, the degree given in \eqref{eq deg of parabolic reduction} is simply the degree of $E$ with respect to $\chi_s,$  and will be denoted $\deg_\chi(E)$. Again, if a multiple $q\cdot\chi_s$ lifts to a character $\tilde \chi:\rG\to\C^*$ we have 
\begin{equation}
  \label{eq deg of central chacacter} 
  \deg_\chi(E)=\frac{1}{q}\deg E(\tilde\chi).
\end{equation}

Let $d\rho:\fg\to\fgl(V)$ be the differential of $\rho$ and let $\fz^\R$ be the center of $\fk^\R$ and $\fk^\R=\fk_{ss}^\R=\fz^\perp$. Define 
\[\fk^\R_\rho=\fk^\R_{ss}\oplus (\ker(d\rho|_{\fz^\R}))^\perp.\]
We are now ready to define $\alpha$-stability notions for $\alpha\in i\fz^\R.$
\begin{definition}\label{def: stability}
Let $\alpha\in i\fz^\R$. A $(\rG,V)$-Higgs pair $(E,\varphi)$ is:
\begin{itemize}
  \item \emph{$\alpha$-semistable} if for any $s\in i\fk^\R$ 
   and any holomorphic reduction $\sigma\in H^0(E(\rG/\rP_s))$ such that  $\varphi\in H^0(E_\sigma(V_s)\otimes K)$, we have $\deg E(\sigma,s)\geq \langle\alpha,s\rangle$.

\item \emph{$\alpha$-stable} if it is $\alpha$-semistable and for any $s\in i\fk^\R_\rho$ and any holomorphic reduction $\sigma\in H^0(E(\rG/\rP_s)$ such that  $\varphi\in H^0(E_\sigma(V_s)\otimes K)$, we have $\deg E(\sigma,s)> \langle\alpha,s\rangle$.
 
\item \emph{$\alpha$-polystable} if it is $\alpha$-semistable and whenever $s\in i\fk^\R$ and $\sigma\in H^0(E(\rG/\rP_s))$  satisfy $\varphi\in H^0(E_\sigma(V_s)\otimes K)$ and $\deg E(\sigma,s)=\langle\alpha,s\rangle$,
 there is a further holomorphic reduction $\sigma'\in H^0(E_\sigma(\rP_s/\rL_s))$ such that $\varphi\in H^0(E_{\sigma'}(V_s^0)\otimes K)$.
\end{itemize}
\end{definition}

The {\bf moduli space of $\alpha$-polystable $(\rG,V)$-Higgs pairs} over $X$ is defined as the set of isomorphism classes of $\alpha$-polystable $(\rG,V)$-Higgs pairs and will be denoted by $\cM^\alpha(\rG,V).$ A GIT construction of these spaces is given by Schmitt in \cite{schmittGITbook} and by Simpson for the moduli space of $0$-polystable $\rG$-Higgs bundles \cite{SimpsonModuli1}.
\begin{remark}
When $\alpha=0$, we refer to $0$-stability simply as stability (similarly for (semi, pol)stability), and denote the moduli space by $\cM(\rG,V).$ 
The {\bf moduli space of polystable $\rG$-Higgs bundles} will be denoted by $\cM(\rG)$, and, for a real form $\rG^\R<\rG,$ the {\bf moduli space of polystable $\rG^\R$-Higgs bundles} will be denoted by $\cM(\rG^\R).$
\end{remark}

There is a natural $\C^*$-action on the moduli spaces of $\alpha$-polystable Higgs pairs given by $\lambda\cdot(E,\varphi)=(E,\lambda\varphi).$ If $\fg=\bigoplus_{j\in\Z}\fg_j$ is a $\Z$-grading, then the $\C^*$-action is trivial on the moduli space of $(\rG_0,\fg_k)$-Higgs pairs by \eqref{eq iso of G0gk pairs}. 
As a result, polystable Higgs bundles which are Hodge bundles define fixed points of the $\C^*$-action on the moduli space of Higgs bundles. Simpson proved the converse \cite{SimpsonVHS,localsystems}. Namely, all $\C^*$-fixed points in the moduli space of Higgs bundles are Hodge bundles. 
In fact, if a Higgs bundle is a Hodge bundle, then polystability of the Higgs bundle is equivalent to polystability of the associated pair. To explain why this is true, we need to use the correspondence between stability and solutions to gauge theoretic equations. 

We first describe this for $\rG^\R$-Higgs bundles. Fix a maximal compact subgroup $\rK^\R<\rG$ and let $\tau:\fg\to\fg$ be the resulting conjugate linear involution. Now fix a real form $\rG^\R<\rG$ such that $\rK^\R\cap\rG^\R=\rH^\R$ is a maximal compact subgroup of $\rG^\R$ and let $\rH<\rG$ be the complexification of $\rH^\R.$ Let $\fg=\fh\oplus\fm$ be the complexified Cartan decomposition of the real form.  
Consider a $\rG^\R$-Higgs bundle $(E,\varphi)$. A metric on $E$ is by definition a structure group reduction $h\in \Gamma(E(\rH/\rH^\R))$ of $E$ to $\rH^\R.$ Associated to a metric $h$, there is a unique connection $A_h$ (the Chern connection) which is compatible with the reduction and the holomorphic structure. Let $E_h$ be the $\rH^\R$-subbundle associated to a metric h. Then $E(\fm)$ is canonically identified with $E_h(\fm)$.
The compact real form $\tau$ defines an anti-holomorphic involution of $E_h(\fm)$. Combining this with conjugation of 1-forms defines an involution 
\[\tau:\Omega^{0,1}(E_h(\fm))\to\Omega^{0,1}(E_h(\fm)).\]
The Higgs field $\varphi\in H^0(E(\fm)\otimes K)$ defines a $(1,0)$-form valued in $E(\fm)$. Thus $[\varphi,-\tau(\varphi)]$ defines a 2-form valued in $E_h[\fh^\R]$, and we can make sense of the following equation 
\begin{equation}\label{eq Hitchin eq}
  F_h+[\varphi,-\tau(\varphi)]=0,
\end{equation}
where $F_h$ is the curvature of the Chern connection. 
These are the {\bf Hitchin equations}. The following theorem relates solutions of the Hitchin equations with stability; it was proven by Hitchin \cite{selfduality} for $\rSL_2\C$ and by Simpson \cite{SimpsonVHS} for complex reductive Lie groups, see \cite{HiggsPairsSTABILITY} for the general statement. 
\begin{theorem}\label{thm kh corr Higgs}
   A $\rG^\R$-Higgs bundle $(E,\varphi)$ is polystable if and only if there a metric $h$ on $E$ which solves the Hitchin equations \eqref{eq Hitchin eq}. 
 \end{theorem} 
\begin{remark}\label{rem flat metric conn}
When the group $\rG^\R$ is compact, then the Hitchin equations are just $F_h=0.$ That is, there is a flat metric on the bundle, and the classical results of Narasimhan-Seshadri \cite{NarasimhanSeshadri} and Ramanathan \cite{ramanathan_1975}.
\end{remark}

For general $(\rG,V)$-Higgs pairs, we fix a maximal compact $\rK^\R<\rG$ and an $\rK^\R$ invariant Hermitian inner-product on V. Now the action of $\rK^\R$ on $V$ is Hamiltonian and has an associated moment map $\mu:V\to(\fk^\R)^*\to\fk^\R$, where we use the inner-product on $\fk^\R$ to identify $(\fk^\R)^*$ with $\fk^\R$. If $\varphi\in \Omega^{1,0}(E(V)),$ then one can define a bundle version of the moment map so that $\mu(\varphi)\in\Omega^{1,1}(E(\fk^\R))$, and polystability of $(\rG,V)$-Higgs pairs is equivalent the existence of a metric $h$ solving the equation
$F_h+\mu(\varphi)=0$ (see \cite{HiggsPairsSTABILITY} for more details).
For $\rG$-Higgs bundles, we fix a compact real form $\rK^\R<\rG$ and let $\tau:\fg\to\fg$ the associated involution. The Hermitian inner-product on $\fg$ is given by $B(\cdot,-\tau(\cdot))$ and the moment map $\mu_\rG:\fg\to\fk^\R$ is given by $\mu_\rG(x)=[x,-\tau(x)]$.

For $(\rG_0,\fg_k)$-Higgs pairs, we choose a compact real form $\rK^\R<\rG$ such that $\tau|_{\fg_0}=\Id.$ Recall from Remark \ref{rem compact inv and Z-grading}, that $\tau(\fg_j)=\fg_{-j}$ for all $j.$ As a result, $\rK_0^\R=\rG_0\cap\rK^\R$ is a compact form of $\rG_0$ and $B(x,-\tau(y))$ defines a $\rK_0^\R$-invariant Hermitian form on $\fg_k.$ 
The moment map $\mu_{\rG_0}:\fg_k\to\fk_0^\R$ is then given by restricting the moment map $\mu_\rG:\fg\to\fk^\R$ and orthogonally projecting onto $\fk_0^\R$. That is, $\mu_{\rG_0}(x)$ is the orthogonal projection of $[x,-\tau(x)]$ onto $\fk_0^\R.$ 
But for $x\in\fg_k,$ orthogonal projection is unnecessary since $[x,-\tau(x)]\in\fk_0^\R$.
Thus, a $(\rG_0,\fg_k)$-Higgs pair $(E,\varphi)$ is polystable if and only if there is a metric $h\in \Gamma(E/\rK_0^\R)$ so that $F_h+[\varphi,-\tau(\varphi)]=0.$ Such a metric solves the Hitchin equations \eqref{eq Hitchin eq} for the associated $\rG$-Higgs bundle $(E(\rG),\varphi).$ 
Conversely, using an averaging argument, Simpson \cite{SimpsonVHS} showed that, if $(E,\varphi)$ is a polystable $\rG$-Higgs bundle which is a Hodge bundle, then the metric solving the Hitchin equations is compatible with the Hodge bundle reduction. We summarize this in a proposition. 

\begin{proposition}
 A $(\rG_0,\fg_k)$-Higgs pair $(E,\varphi)$ is polystable as a $(\rG_0,\fg_k)$-Higgs pair if and only if the associated $\rG$-Higgs bundle $(E(\rG),\varphi)$ is polystable as a $\rG$-Higgs bundle. In particular, there is a well defined map of moduli spaces
 \[\cM(\rG_0,\fg_k)\to\cM(\rG)\]
whose image consists of $\C^*$-fixed points, and every $\C^*$-fixed point in $\cM(\rG)$ is in the image of such a map for some $(\rG_0,\fg_k).$
\end{proposition}
\begin{remark}
For stable and simple  Higgs bundles, the associated type of Hodge bundle is unique.
\end{remark}
\begin{remark}
 Note that for gradings $\fg=\fg_{-1}\oplus\fg_0\oplus\fg_1$, $(\rG_0,\fg_{-1}\oplus\fg_1)$-pair stability agrees with Higgs bundle stability. 
This is not true for other gradings. 
 Namely, $(\rG_0,\bigoplus_{j\neq0}\fg_j)$ or $(\rG_0,\fg_{-1}\oplus\bigoplus_{j>0}\fg_j)$-pair stability does not imply Higgs bundle stability. 
\end{remark}
 \begin{example}\label{ex uniformizing metric}
  The $\rPSL_2\C$-Higgs bundle $(E_\rT,e)$ from Example \ref{ex sl2 uniformizing} is polystable. Since $E_\rT$ is the frame bundle of $K^{-1},$ a metric on $E_\rT$ defines a metric on the Riemann surface $X.$ A solution to the Hitchin equations in this case is equivalent to a metric of constant curvature on $X$ (see \cite{selfduality}). As a result, the Higgs bundle $(E_\rT,e)$ will be referred to as the uniformizing Higgs bundle for $X.$ 
 \end{example}




\subsection{Character varieties and variations of Hodge structure}\label{sec:char-vari-vari}
Given a polystable $\rG^\R$-Higgs bundle $(E,\varphi)$, there is a metric $h$ on $E$ solving the Hitchin equations \eqref{eq Hitchin eq}. For such a metric $h,$ the connection 
\[D=A_h+\varphi-\tau(\varphi)\]
defines a flat connection on the $\rG^\R$-bundle $E_h(\rG^\R)$, where $E_h$ is the $\rH^\R$-bundle associated to the metric $h.$ As a result, a polystable $\rG^\R$-Higgs bundle $(E,\varphi)$ defines representations $\rho:\pi_1(X)\to\rG^\R$ such that $E_h(\rG^\R)\cong \widetilde X\times_\rho\rG^\R$.

Given a representation $\rho:\pi_1(X)\to\rG^\R$, a metric $h_\rho$ on a flat bundle $\widetilde X\times_\rho\rG^\R$ can be interpreted as a $\rho$-equivariant map to the Riemannian symmetric space of $\rG^\R:$ 
\[h_\rho:\widetilde X\to\rG^\R/\rH^\R.\]
A metric $h_\rho$ is called harmonic if it is a critical point of the energy function 
\[\cE(h_\rho)=\frac{1}{2}\int_X|dh_\rho|^2.\]
This makes sense since, for two dimensional domains, the energy only depends on the conformal structure of the domain.
It turns out that a metric $h$ solves the Hitchin equations \eqref{eq Hitchin eq} if and only if the $\rho$-equivariant map $h_\rho:\tilde X\to\rG^\R/\rH^\R$ is harmonic. 

\begin{remark}\label{rem higgs field is}
  In this correspondence, the Higgs field $\varphi$ is identified with the $(1,0)$-part of the differential of the map $h_\rho$. As a result, for the uniformizing Higgs bundle $(E_\rT,e)$, the harmonic metric $h_\rho:\widetilde X\to\mathbb{H}^2$ is $\rho$-equivariant biholomorphism. Thus, $X=\mathbb{H}^2/\rho(\pi_1(X))$ and $\rho$ is the uniformizing representation of the Riemann surface. 
\end{remark}

A representation $\rho:\pi_1(X)\to\rG^\R$ is called reductive if post composing $\rho$ with the adjoint representation of $\rG^\R$ decomposes as a direct sum of irreducible representations. 
Corlette's theorem \cite{canonicalmetrics} (proven by Donaldson \cite{harmoicmetric} for $\rSL_2\C$) asserts that given a representation $\rho:\pi_1(X)\to\rG^\R$, there exists a $\rho$-equivariant harmonic map $h_\rho:\widetilde X\to\rG^\R/\rH^\R$ if and only if $\rho$ is reductive. 
Denote the set of reductive representations $\rho:\pi_1(X)\to\rG^\R$ by $\Hom^+(\pi_1(X),\rG^\R)$. The moduli space of $\rG^\R$-conjugacy classes of  of representations $\pi_1(X)$ in $\rG^\R$ is called the {\bf character variety} and denoted by 
\[\cR(\rG^\R)=\Hom^+(\pi_1(X),\rG^\R)/\rG^\R.\]

Combining Corlette's theorem with the Hitchin--Kobayashi correspondence defines a homeomorphism $\cM(\rG^\R)\cong\cR(\rG^\R)$ between the moduli space  of $\rG^\R$-Higgs bundles on $X$ and the $\rG^\R$-character variety. This is called the {\bf nonabelian Hodge correspondence}. 

Let $\rho:\pi_1(X)\to\rG^\R$ be a representation and $\tilde\rG^\R<\rG^\R$ be a reductive subgroup, we say that $\rho$ factors through $\tilde\rG^\R$ if $\rho$ can be written as $\rho:\pi_1(X)\to\tilde\rG^\R\hookrightarrow\rG^\R.$ The following proposition is immediate from the nonabelian Hodge correspondence.
\begin{proposition}\label{prop equiv to reduction}
  Let $\tilde\rG^\R<\rG^\R$ be a reductive subgroup with maximal compact $\tilde\rH^\R<\rH^\R$ and let $\rho:\pi_1(X)\to\rG^\R$ be a reductive representation. The following are equivalent 
  \begin{itemize}
    \item $\rho$ factors through $\tilde\rG^\R$,
    \item the $\rho$-equivariant harmonic map $h_\rho$ factors as $h_\rho:\widetilde X\to \tilde\rG^\R/\tilde\rH^\R\hookrightarrow \rG^\R/\rH^\R$, and
    \item the associated $\rG^\R$-Higgs bundle reduces to a polystable $\tilde\rG^\R$-Higgs bundle.
  \end{itemize} 
\end{proposition}

We now describe the special properties of the representations and harmonic maps arising from $\C^*$-fixed points in the Higgs bundle moduli space. Let $\rG^\R<\rG$ be a real form of Hodge type with maximal compact $\rH^\R$ and complexified Cartan decomposition $\fg=\fh\oplus\fm.$  Fix a $\Z$-grading $\fg=\bigoplus_{j\in\Z}\fg_j$ such that $\fg_{2j}\subset\fh$ and $\fg_{2j+1}\subset\fm$ for all $j$. 
Recall that the homogeneous space $D=\rG^\R/\rH_0^\R$ is a period domain. The holomorphic tangent bundle decomposes as 
\[T^{1,0}_\C D =\bigoplus_{j>0} (T^{1,0}_\C D)_j=\bigoplus_{j>0}\rG^\R\times_{\rH_0^\R}\fg_j.\]
\begin{definition}
  Fix a period domain $D=\rG^\R/\rH_0^\R$. A variation of Hodge structure over $X$ is a pair $(\rho,f_\rho)$, where $\rho:\pi_1(X)\to\rG^\R$ is a representation and $f_\rho:\widetilde X\to D$ is a $\rho$-equivariant holomorphic map such that 
  \[\partial f_\rho(K^{-1})\subset (T^{1,0}_\C D)_1.\]
\end{definition}
The following proposition describes the representations and harmonic maps associated to polystable Higgs bundles fixed by the $\C^*$-action.
\begin{proposition}
  Let $\rho:\pi_1(X)\to\rG^\R$ be a reductive representation. Then the Higgs bundle associated to $\rho$ is a $\C^*$-fixed point if and only if $\rho$ factors through a reductive subgroup $\tilde\rG^\R<\rG^\R$ of Hodge type, and there is a period domain $\tilde\rG^\R/\tilde\rH_0^\R$ such that the associated $\rho$-equivariant harmonic map $h_\rho:\widetilde X\to\tilde\rG^\R/\rH^\R$ lifts to a variation of Hodge structure 
  \[\xymatrix@R=1em{&\tilde\rG^\R/\tilde\rH_0^\R\ar[d]&\\\widetilde X\ar[r]_{h_\rho}\ar[ur]^{f_\rho}&\tilde\rG^\R/\tilde\rH^\R\ar[r]&\rG^\R/\rH^\R}\]
\end{proposition}


\section{Toledo invariant and Arakelov--Milnor inequality}

Let $\rG$ be a complex semisimple Lie group and $\fg=\bigoplus_{j\in\Z}\fg_j$ be a $\Z$-grading  with grading element $\zeta\in\fg_0$ and let $\rG_0<\rG$ be the centralizer of $\zeta.$
Consider the phvs $(\rG_0,\fg_1)$ and recall from \eqref{def Toledo character} that the  Toledo character $\chi_T:\fg_0\to\C$ is given by $\chi_T(x)=B(\zeta,x)B(\gamma,\gamma)$. Recall that the degree of a bundle with respect to a character is defined by \eqref{eq deg of central chacacter}.

\begin{definition}\label{def toledo inv}
  Let $(E,\varphi)$ be a $(\rG_0,\fg_1)$-Higgs pair and $\chi_T:\rG_0\to\C^*$ be the Toledo character. Then the Toledo invariant $\tau(E,\varphi)$ is defined to by 
  \[\tau(E,\varphi)=\deg_{\chi_T}(E).\]
\end{definition}

\begin{remark} 
For the grading $\fg=\fg_{-1}\oplus\fg_0\oplus\fg_1$, the canonical real form $\rG^\R<\rG$ from Proposition \eqref{prop hodge real form of Zgrad} is a Hermitian Lie group. For such real forms, a $\rG^\R$-Higgs bundle is an $(\rG_0,\fg_{-1}\oplus\fg_1)$-Higgs pair. The Toledo invariant of such a Higgs bundle agrees with the Toledo invariant defined above \cite{BGRmaximalToledo}.  
\end{remark}

 Recall Definition \ref{def toledo rank} of the Toledo rank $\rk_{\chi_T}(x)$ of a point $x\in\fg_1$. We define the {\bf Toledo rank} $\rk_T(\varphi)$ of a $(\rG_0,\fg_1)$-Higgs pair $(E,\varphi)$ to be 
 \[\rk_T(\varphi)=\rk_T(\varphi(x))\ \text{ for a generic $x\in X$}.\]
\begin{theorem}\label{thm: Arakelov--Milnor ineq}
  Let $\rG$ be a complex semisimple Lie group and $\fg=\bigoplus_{j\in\Z}\fg_j$ be a $\Z$-grading with grading element $\zeta\in\fg_0$. Let $\rG_0<\rG$ be the centralizer of $\zeta$ and $\alpha=\lambda\zeta$ for $\lambda\in\R.$ If $(E,\varphi)$ is an $\alpha$-semistable $(\rG_0,\fg_1)$-Higgs pair, then the Toledo invariant $\tau(E,\varphi)$ satisfies the following inequality
  \[-\rk_T(\varphi)(2g-2)+\lambda(B(\gamma,\gamma)B(\zeta,\zeta)-\rk_T(\varphi))\leq ~ \tau(E,\varphi)\leq \lambda B(\gamma,\gamma) B(\zeta,\zeta).\]
\end{theorem}
Before proving the theorem we list some immediate consequences. 
\begin{corollary}
Assume $(\rG_0,\fg_1)$ is a JM-regular phvs and set $\alpha=\lambda\zeta$. Then the Toledo invariant of a $\alpha$-semistable $(\rG_0,\fg_1)$-Higgs pair $(E,\varphi)$ satisfies 
 \[-\rk_T(\varphi)(2g-2)+\lambda(\rk_T(\rG_0,\fg_1)-\rk_T(\varphi))\leq \tau(E,\varphi)\leq \lambda\rk(\rG_0,\fg_1).\]
\end{corollary}
For the application to  Higgs bundles fixed by the $\C^*$-action we are interested in the case $\alpha=0.$
\begin{corollary}\label{cor am ineq for comp}
  For $\alpha=0$, the Toledo invariant $\tau(E,\varphi)$ of a polystable $(\rG_0,\fg_1)$-pair satisfies the inequality $-\rk_T(\varphi)(2g-2)\leq \tau(E,\varphi)\leq 0.$ In particular,
  \[|\tau(E,\varphi)|\leq \rk_T(\rG_0,\fg_1)(2g-2).\]
\end{corollary}
\begin{proof}
  This is a direct application of Theorem \ref{thm: Arakelov--Milnor ineq} in the case $\lambda=0$, for which in fact it is enough to assume the semistability 
of $(E,\varphi)$. However, with the extra assumption of polystability, the result can also be deduced from our holomorphic sectional curvature computations in Proposition \ref{prop:sectional_curvature}, as we shall now explain. By Theorem \ref{thm kh corr Higgs} such $(E,\varphi)$ admits a solution $h$ to the Hitchin equations (\ref{eq Hitchin eq}), so we have a variation of Hodge structure as studied in §~\ref{sec:char-vari-vari}. In particular we obtain an equivariant horizontal holomorphic map $f:\tilde X\rightarrow \rG^\R/\rH_0^\R$. Denote by $\omega$ the 2-form associated to the Hermitian metric on $\rG^\R/\rH_0^\R$ considered in §~\ref{sec hol sec curv}, normalized to have horizontal holomorphic sectional curvature $-1\leq K\leq -\frac1{\rk_T(\rG_0,\fg_1)}$.
  By the Schwarz Lemma, one has \[ \frac1{2\pi} \int_X f^*\omega \leq \rk_T(\rG_0,\fg_1) (2g-2).\] The result will then follow from the identification \[ f^*\omega = - \chi_T(iF_h), \]
  where $F_h$ is the curvature as in (\ref{eq Hitchin eq}). From (\ref{eq Hitchin eq}) we have $\chi_T(F_h)=-\chi_T([\varphi,\varphi^*])$; writing locally $\varphi=a dz$, we obtain
  \begin{align*}
  \chi_T(F_h)&=-B(\zeta,[a,a^*])B(\gamma,\gamma)dz\wedge d\bar z\\ &=-B([\zeta,a],a^*)B(\gamma,\gamma)dz\wedge d\bar z \\ &=-B(a,a^*)B(\gamma,\gamma)dz\wedge d\bar z\\ &=-\langle a,a\rangle dz\wedge d\bar z.
  \end{align*}
  On the other hand, $f^*\omega = \omega(a,\bar a) dz\wedge d\bar z = i \langle a,a \rangle dz\wedge d\bar z$. Therefore $f^*\omega=-\chi_T(iF_h)$ which completes the proof.

  The same argument gives the more precise inequality $\tau(E,\varphi) \geq -\rk_T(\varphi)(2g-2)$, starting from the holomorphic curvature estimate (\ref{eq:1}).
\end{proof}

Since the Toledo invariant has discrete values, it is constant on the connected components of the moduli space of polystable $(\rG_0,\fg_1)$-Higgs pairs $\cM(\rG_0,\fg_1)$. As a result, we have the following decomposition of $\cM(\rG_0,\fg_1).$ 
\begin{corollary}
  The moduli space of polystable $(\rG_0,\fg_1)$-Higgs pairs decomposes as 
  \[\cM(\rG_0,\fg_1)=\coprod\limits_{-\rk_T(\rG_0,\fg_1)(2g-2)\leq \tau\leq 0}\cM^\tau(\rG_0,\fg_1),\]
  where $(E,\varphi)\in\cM^\tau(\rG_0,\fg_1)$ if and only if $\tau(E,\varphi)=\tau.$
\end{corollary}
\begin{remark}
  We will discuss the special properties of the space $\cM^{\tau}(\rG_0,\fg_1)$ for the minimal value $\tau=-(2g-2)\rk_T(\rG_0,\fg_1)$ in the next section.
\end{remark}

\begin{proof}[Proof of Theorem \ref{thm: Arakelov--Milnor ineq}] 
The upper bound follows immediately from stability. Note that $s=-B(\gamma,\gamma)\zeta$ is in the center of $\fg_0$ and has the property that $\Ad(e^{ts})(x)$ is bounded as $t\to\infty$ for all $x\in\fg_1.$ Since $s$ is in the center of $\fg_0,$ the associated parabolic is all of $\rG_0.$ Hence, $\alpha$-semistability of $(E,\varphi)$ implies 
\[\deg E(-B(\gamma,\gamma)\zeta)\geq B(\alpha, -B(\gamma,\gamma)\zeta).\]
Since $-\tau(E,\varphi)=\deg E(-B(\gamma,\gamma)\zeta)$ and  $\alpha=\lambda\zeta$ we have $\tau(E,\varphi) \leq \lambda B(\gamma,\gamma)B(\zeta,\zeta).$

The lower bound is more interesting. For generic $x\in X$, $\varphi(x)\in\fg_1$ is in a fixed $\rG_0$-orbit. Let $e$ be any element in this orbit and $\{f,h,e\}$ be an associated $\fsl_2$-triple with $h\in\fg_0$. Let $(\hat\rG_0,\hat\fg_1)$ be the associated maximal JM-regular phvss of $(\rG_0,\fg_1)$ for $e$ (see Definition \ref{def max JM reg}). By Proposition \ref{prop max jm reg for e}, $e\in\hat\Omega\subset\fg_1$ is in the open $\hat\rG_0$-orbit. 
 Recall from Proposition \ref{prop: s defines parabolic of g0} that $e$ defines a parabolic subalgebra $\fp_{0,e}\subset\fg_0$ and a corresponding parabolic subgroup $\rP_{0,e}<\rG_0$. Since the construction of $\rP_{0,e}$ is canonical, we obtain a holomorphic reduction of structure group of $E$ to $\rP_{0,e}$ on the open subset of $X$ where $\rk_T(\varphi(x))=\rk_T(e)$. This can be extended to a reduction over all of $X.$ Let $E_\sigma$ be the associated $\rP_{0,e}$-bundle.

Set $\hat\zeta=\frac{h}{2}$ and write $\zeta=\hat\zeta+s.$ Recall from Proposition \ref{prop: s defines parabolic of g0} that the parabolic subalgebra $\fp_{0,e}$ is given by 
  \[\fp_{0,e}=\{x\in\fg_0~|~\Ad(e^{ts})x \text{ is bounded as $t\to\infty$}\}.\]
By definition of the Toledo character, we have 
\[\chi_T(x)=B(\gamma,\gamma)(B(\hat\zeta,x)+B(s,x))\]
Thus
\begin{equation}\label{eq diff of degree}
  \deg E(\sigma,B(\gamma,\gamma)s)=\tau(E,\varphi)-\deg_{\hat\chi_T}(E_\sigma),
\end{equation}
where $\hat\chi_T:\fp_{0,e}\to\C$ is the character defined by $\hat\chi_T(x)=B(\gamma,\gamma)B(\hat\zeta,x).$ 

The Levi factor of $P_{0,e}$ is isomorphic to $\hat\rG_0$. Hence projecting onto the Levi factor defines a holomorphic reduction $E_{\sigma}(\hat\rG_0)$ with $\varphi\in H^0(E_{\sigma}(\hat\rG_0)(\hat\fg_1)\otimes K)$. Recall from Proposition \ref{prop JM reg toledo has rel invar} that there is a positive integer $q$ such that the character $q\cdot \hat\chi_T$ lifts to a character of $\hat\chi_{T,q}:\hat\rG_0\to\C^*$ which has a relative invariant $F:\hat\fg_1\to\C$ of degree $q\cdot \hat\chi_T(\hat\zeta)$. 
Since $\varphi$ is generically in the open $\hat\rG_0$-orbit of $\hat\fg_1,$ applying the relative invariant to the Higgs field defines a nonzero holomorphic section of a line bundle $E_\sigma(\hat\chi_{T,q})\otimes K^{q\cdot \hat\chi_T(\hat\zeta)}$
\[F(\varphi)\in H^0(E_\sigma(\hat\chi_{T,q})\otimes K^{q\cdot \hat\chi_T(\hat\zeta)})\setminus\{0\}.\]
Note that $\deg_{\hat\chi_T}(E_\sigma)=\frac{1}{q}\deg(E_\sigma(\hat\chi_{T,q}))$. Thus, we have 
\begin{equation}
  \label{eq 2g-2}
  \deg_{\hat\chi_T}(E_\sigma)\geq -\hat\chi_T(\hat\zeta)(2g-2).
\end{equation}

We now use the assumption that $(E,\varphi)$ is $\alpha$-semistable for $\alpha=\lambda\zeta.$ Since $[s,\hat\fg_1]=0,$ we have $\hat\fg_1\subset\fg_1^s.$ Thus, we have $\varphi\in\ H^0(E_{\sigma}(\hat\fg_1^s)\otimes K).$ By the definition of the stability for $\alpha=\lambda\zeta$, we have 
\[B(\lambda\zeta, B(\gamma,\gamma)s)\leq \deg E(\sigma, B(\gamma,\gamma)s).\]
Combining this with \eqref{eq diff of degree} and \eqref{eq 2g-2} gives an inequality for the Toledo invariant:
\[-\hat\chi_T(\hat\zeta)(2g-2)+\lambda B(\gamma,\gamma) B(\zeta, s)\leq \tau(E,\varphi).\]

To finish the proof, recall that $B(\hat\zeta,s)=0$,  and $B(\gamma,\gamma)B(\zeta,\hat\zeta)=\rk_T(\varphi).$ Thus, 
\[-\hat\chi_T(\hat\zeta)= B(\gamma,\gamma)B(\hat\zeta,\hat\zeta)=B(\gamma,\gamma)B(\zeta,\hat\zeta)=\rk_T(\varphi),\]
and 
\[\lambda B(\gamma,\gamma)B(\zeta,s)=\lambda(B(\gamma,\gamma)B(\zeta,\zeta)-B(\gamma,\gamma)B(\zeta,\hat\zeta))=\lambda(B(\gamma,\gamma)B(\zeta,\zeta)-\rk_T(\varphi)).\]
Hence we obtain the desired inequality
\[-\rk_T(\varphi)(2g-2)+\lambda(B(\gamma,\gamma)B(\zeta,\zeta)-\rk_T(\varphi))\leq \tau(E,\varphi).\]
 \end{proof}


\section{Rigidity results for maximal Hodge bundles and Variations of Hodge structure}\label{sec rigidity}
For this section, let $\rG$ be a complex semisimple Lie group with Lie algebra $\fg$ and let $X$ be a compact Riemann surface of genus $g\geq 2$ and canonical bundle $K.$ 
Fix a $\Z$-grading $\fg=\bigoplus_{j\in\Z}\fg_j$ with grading element $\zeta$ and let $\rG_0<\rG$ be the centralizer of $\zeta.$

We will call a polystable $(\rG_0,\fg_1)$-Higgs pair $(E,\varphi)$ {\bf maximal} if the absolute value of the Toledo invariant is maximized.
By Corollary \ref{cor am ineq for comp}, the Toledo invariant $\tau(E,\varphi)$ of a polystable $(\rG_0,\fg_1)$-Higgs pair on $X$ satisfies the inequality 
\[-\rk_T(\rG_0,\fg_1)(2g-2)\leq \tau(E,\varphi)\leq0.\]
 Thus, $(E,\varphi)$ is maximal if and only if $\tau(E,\varphi)=-\rk_T(\rG_0,\fg_1)(2g-2).$

\subsection{The JM-regular case}
We start with the JM-regular case. Assume $(\rG_0,\fg_1)$ is a JM-regular phvs. Fix $e\in\Omega$ and let $\{f,h,e\}=\{f,2\zeta,e\}$ be the associated $\fsl_2$-triple.  Let $\rS<\rG$ be the connected subgroup with Lie algebra spanned by $\{f,h,e\}.$ Note that $\rS$ is isomorphic to $\rPSL_2\C$ or $\rSL_2\C$ depending on $\rG$ and the $\fsl_2$-triple. Finally, let $\rC<\rG$ be the reductive group which centralizes the $\fsl_2$-triple $\{f,h,e\}.$ Note that $\rC=\rG_0^e<\rG_0$ is also the $\rG_0$-stabilizer of $e$.

Let $\rT<\rS$ be the connected subgroup with Lie algebra $\langle h\rangle,$ and note that $\rT<\rG_0.$
Recall the uniformizing Hodge bundle $(E_\rT,e)$ for $X$ from Examples \ref{ex sl2 uniformizing} and \ref{ex uniformizing metric}. Here $E_\rT$ is the holomorphic frame bundle of $K^{-1}$ if $\rS=\rP\rSL_2\C$ and the holomorphic frame bundle of $K^{-\frac{1}{2}}$ if $\rS=\rSL_2\C.$ Since $E_\rT(\langle e\rangle)\otimes K\cong \cO$, the Lie algebra element $e$ defines a holomorphic section of $E_\rT(\langle e\rangle) \otimes K.$ Moreover, $(E_\rT,e)$ is polystable, and a solution to the Hitchin equations is equivalent to a metric of constant curvature on $X.$ Extending the structure group to $\rG_0$ defines a maximal polystable $(\rG_0,\fg_1)$-Higgs pair.

\begin{proposition}
  The $(\rG_0,\fg_1)$-Higgs pair $(E_\rT(\rG_0),e)$ is polystable and maximal.
\end{proposition}
\begin{proof}
  First note that $(E,\varphi)=(E_\rT(\rG_0),e)$ is polystable since $(E_\rT,e)$ polystable. 
  By construction, we have $\tau(E,\varphi)=-\rk_{T}(\rG_0,\fg_1)(2g-2).$
\end{proof}

Other examples of maximal $(\rG_0,\fg_1)$-Higgs pairs are given by twisting the above uniformizing Higgs bundle by a holomorphic $\rC$-bundle. 
Let $E_\rC$ be a holomorphic $\rC$-bundle on $X$. Since $\rT$ and $\rC$ are commuting subgroups of $\rG_0$, the multiplication map $m:\rT\times\rC\to\rG_0$ is a group homomorphism. Thus, we can form a $\rG_0$-bundle 
\[E_\rT\otimes E_\rC(\rG_0)=(E_\rT\times E_\rC)\times_m \rG_0.\]
This is the principal bundle version of the tensor product of vector bundles. 
As with vector bundles, given metrics $h_\rT$ and $h_\rC$ on $E_\rT$ and $E_\rC$ respectively, there is a metric $h_\rT\otimes h_\rC$ on $E_\rT\otimes E_\rC$. Similarly, given connections $A_\rT$ and $A_\rC$ on $E_\rT$ and $E_\rC$ respectively, $A_\rT+A_\rC$ defines a connection on $E_\rT\otimes E_\rC$ and the curvature satisfies
\[F_{A_\rT+A_\rC}=F_{A_\rT}+F_{A_\rC}.\]

Since $\rC$ acts trivially on $\langle e\rangle,$ we have 
\[e\in H^0((E_\rT\otimes E_\rC)(\fg_1)\otimes K).\]
Also, since the Lie algebra $\fc$ is perpendicular to $\zeta$, the Toledo invariant is unchanged
\[\tau((E_\rT\otimes E_\rC)(\rG_0),e)=\tau(E_\rT(\rG_0),e).\]
Moreover, if $E_\rC$ is a polystable $\rC$-bundle\footnote{In particular, $E_\rC$ must have degree zero with respect to any character $\chi:\fc\to\C.$}, then there is a metric $h_\rC$ on $E_\rC$ whose associated Chern connection is flat (see Remark \ref{rem flat metric conn}). As a result, if $h_\rT$ is a metric on $E_\rT$ solving the Hitchin equations for $(E_\rT,e)$, then $h_\rT\otimes h_\rC$ solves the Hitchin equations for the Higgs bundle $(E_\rT\otimes E_\rC(\rG),e)$. 
Hence, if $E_\rC$ is a polystable $\rC$-bundle, then $(E_\rT\otimes E_\rC(\rG_0),e)$ is a polystable $(\rG_0,\fg_1)$-Higgs pair. 

So far, we have shown that there is a well defined map 
\begin{equation}\label{eq psie}
  \Psi_e:\xymatrix@R=0em{\cN(\rC)\ar[r]&\cM^{\max}(\rG_0,\fg_1)\\
E_\rC\ar@{|->}[r]&(E_\rT\otimes E_\rC(\rG),e)}
\end{equation}
from the moduli space $\cN(\rC)$ of polystable $\rC$-bundles to the moduli space of polystable $(\rG_0,\fg_1)$-Higgs pairs which are maximal. The map $\Psi_e$ from \eqref{eq psie} is a moduli space version of a restriction of the global Slodowy slice map from \cite{ColSandGlobalSlodowy} and generalizes a restriction of the Cayley correspondence in the case of a $\Z$-grading defining a Hermitian group \cite{BGRmaximalToledo}.
\begin{theorem}\label{thm rigidity jm regular}
  The map $\Psi_e:\cN(\rC)\to\cM^{\max}(\rG_0,\fg_1)$ from \eqref{eq psie} defines an isomorphism between the moduli space of polystable $\rC$-bundles and the moduli space of polystable maximal $(\rG_0,\fg_1)$-Higgs pairs.
\end{theorem}
\begin{remark}
 For classical Lie groups, the centralizers $\rC<\rG$ of $\fsl_2$-triples is easily described in terms of partitions \cite{centralizerSpringerSteinberg} (see \cite[Theorem 6.1.3]{CollMcGovNilpotents}).
\end{remark}

\begin{proof}
  We first show that $\Psi_e$ is surjective. Let $(E,\varphi)$ be a polystable maximal $(\rG_0,\fg_1)$-Higgs pair. As above, let $E_\rT$ be the holomorphic frame bundle of $K^{-1}$ or $K^{-\frac{1}{2}}$ if $\rS$ is isomorphic to $\rPSL_2\C$ or $\rSL_2\C$ respectively.  Note that 
  \[E(\fg_1)\otimes K\cong (E_\rT^{-1}\otimes E)(\fg_1).\]  
Since $(\rG_0,\varphi)$ is JM-regular, we have $\rk_T(\rG_0,\fg_1)=\chi_T(\zeta),$ where $\zeta=\frac{h}{2}.$ 
Also, by Proposition \ref{prop JM reg toledo has rel invar}, a positive multiple of the Toledo character $q\cdot\chi_T(x)=qB(\gamma,\gamma)B(\zeta,x)$ lifts to a character $\chi_{T,q}:\rG_0\to\C^*$ which has a relative invariant $F:\fg_1\to\C$ of degree $q\cdot \rk_T(\rG_0,\fg_1).$ As in the proof of Theorem \ref{thm: Arakelov--Milnor ineq}, applying the relative invariant $F$ to the Higgs field defines a nonzero holomorphic section 
\[F(\varphi)\in H^0(E(\chi_{T,q})\otimes K^{q\rk_{T}(\rG_0,\fg_1)}).\]
The degree of $E(\chi_{T,q})\otimes K^{q\rk_{T}(\rG_0,\fg_1)}$ is $q(\tau(E,\varphi)+\rk_T(\rG_0,\fg_1)(2g-2))=0.$ Thus, $F(\varphi)$ is nowhere vanishing, and $\rk_T(\varphi(x))=\rk_T(\rG_0,\fg_1)$ for all $x\in X.$ That is, $\varphi(x)$ is in the open orbit $\Omega$ for all $x\in X.$

Thus the Higgs field $\varphi$ is defines a holomorphic section of $E(\Omega)\otimes K$. But $\Omega=\rG_0/\rC$ and $E(\Omega)\otimes K\cong (E_\rT^{-1}\otimes E)(\rG_0/\rC).$ Thus, the Higgs field defines a holomorphic reduction of structure group of $(E_\rT^{-1}\otimes E)$ to $\rC.$ Let $E_\rC$ be the resulting holomorphic $\rC$-bundle. Twisting both sides by $E_\rT$, we obtain an isomorphism 
\[E\cong (E_\rT\otimes E_\rC)(\rG_0).\]
By construction, $\varphi\in H^0((E_\rT\otimes E_\rC)(\langle e\rangle)\otimes K)$. Since a $(\rG_0,\fg_1)$-Higgs pairs $(E,\varphi)$ is isomorphic to $(E,\lambda\varphi)$ for all $\lambda\in\C^*,$ we can take $\varphi=e.$

To complete the proof of surjectivity, we show that if $(E,\varphi)=((E_\rT\otimes E_\rC)(\rG_0),e)$ is a (poly,semi)stable $(\rG_0,\fg_1)$-Higgs pair, then $E_\rC$ is a (poly,semi)stable bundle.  
Suppose that we have a reduction $(E_{\rP'_s},\sigma')$ of the structure group of $E_\rC$ to the parabolic subgroup $\rP'_s\subset \rC$ 
defined by a $s\in i\fc^\R$, where $\fc^\R$ is the Lie algebra of the maximal compact subgroup $\rC^\R$ of $\rC$. 
The element $s$ defines a parabolic subgroup $\rP_s\subset \rG_0$ as well, and using the map
  $\rC/\rP'_s\to \rG_0/\rP_s$ we obtain from $\sigma'$ a reduction $\sigma$ of the structure group of $E$ to $\rP_s$, resulting in a $\rP_s$-bundle
  $E_{\rP_s}$. Since $\rC$ stabilizes $e\in\fg_1$ we have $e\in \fg_{1,s}^0$.
Since $\langle s,h\rangle=0$, there is no contribution coming from the twisting $E\otimes E_\rT^{-1}$ in the computation of 
$\deg E(\sigma, s)$. Hence $\deg(E)(\sigma,s)=\deg(E_\rC)(\sigma',s)$, and (semi)stability of $(E,\varphi)$ implies the (semi)stability of $E_\rC$. 
For polystability, one must just check additionally that in the equality case, reduction for the Levi subgroup $\rL_s\subset \rP_s$ to a $E_{\rL_s}\subset E$ implies reduction for the Levi subgroup $\rL'_s\subset \rP'_s$, but it is sufficient to take $E_{\rL'_s}=(E_{\rL_s}\otimes E_{\rT}^{-1})\cap E_\rC$. 

For injectivity, note that two polystable $\rC$-bundles $E_\rC$ and $E_\rC'$ define the same point in $\cM^{\max}(\rG_0,\fg_1)$ if there is a holomorphic isomorphism of $\rG_0$-bundles
\[\Phi:E_\rC\otimes E_\rT(\rG_0)\to E_{\rC}'\otimes E_\rT(\rG_0)\]
such that $\Phi^*e=e.$ Since the $\rG_0$ stabilizer of $e\in\fg_1$ is $\rC,$ we conclude that $\Phi$ induces an isomorphism between $E_\rC$ and $E_\rC'.$ 
Note that injectivity implies that the automorphism group of $\Psi_e(E_\rC)$ is equal to the automorphism group of the polystable $\rC$-bundle $E_\rC.$ 
\end{proof}
The component count for the moduli space of polystable $\rC$-Higgs bundles is given the number of topologically distinct $\rC$-bundles on $X$ which have degree zero with respect to any character $\chi:\fc\to\C.$ For example, when $\rC$ is connected and $\Gamma<\pi_1(\rC)$ is the torsion subgroup, then $\cN(\rC)$ has $|\Gamma|$ connected components. As an immediate corollary, we have the following component count.
\begin{corollary}
  We have $\pi_0(\cM^{\max}(\rG_0,\fg_1))=\pi_0(\cN(\rC))$. In particular, the component count is given by the number of topological degree zero $\rC$-bundles on $X.$
\end{corollary}

We now use Theorem \ref{thm rigidity jm regular} to deduce rigidity results for variations of Hodge structure. Consider the complex linear involution $\theta:\fg\to\fg$ which is $(-1)^j\Id$ on each graded factor $\fg_j.$ Recall from \S\ref{sec real forms} that a compact real form $\tau:\fg\to\fg$ satisfies $\tau(\fg_j)=\fg_{-j}.$ Moreover, we can choose such a $\tau$ so that the $\fsl_2$-triple $\{f,h,e\}$ satisfies $f=-\tau(e).$ Let $\rG^\R<\rG$ be the associated real form of Hodge type. 
By construction, the subgroup $\rS<\rG$ defines a subgroup $\rS^\R<\rG^\R$ which is isomorphic to $\rP\rSL_2\R$ if $\rS\cong \rPSL_2\C$ and isomorphic to $\rSL_2\R$ if $\rS\cong\rSL_2\C.$
Also by construction, the $\rG^\R$-centralizer of $\rS^\R$ is the compact real form $\rC^\R<\rC$ of the centralizer of $\rS.$ 

Since $\rC^\R$ and $\rS^\R$ are commuting subgroups of $\rG^\R$, given representations $\rho_1:\pi_1(X)\to\rS^\R$ and $\rho_2:\pi_1(X)\to\rC^\R$, we can form a new representation, multiplying the images
\[\rho_1*\rho_2:\pi_1(X)\to\rG^\R,\ \ (\rho_1*\rho_2)(\gamma)=\rho_1(\gamma)\cdot\rho_2(\gamma).\]
Recall that $(\rG_0,\fg_1)$ is JM-regular if and only if it arises from an even $\fsl_2$-triple. 

\begin{theorem}\label{thm jm reg rigidity of reps for maximal}
  Let $\rG$ be a complex semisimple Lie group with Lie algebra $\fg$. Let $\{f,h,e\}\subset\fg$ be an even $\fsl_2$-triple and $\fg=\bigoplus_{j\in\Z}\fg_j$ be the associated $\Z$-grading with grading element $\frac{h}{2}$, let $\rS<\rG$ be the associated connected subgroup and let $\rC<\rG$ be the centralizer of $\{f,h,e\}.$ 
 Suppose $\rho:\pi_1(X)\to\rG$ is a reductive representation. The Higgs bundle associated to $\rho$ is a maximal Hodge bundle of type $(\rG_0,\fg_1)$ if and only if, up to conjugation, $\rho=\rho_u *\rho_\rC:\pi_1(X)\to\rG,$
where 
\begin{itemize}
  \item $\rho_u:\pi_1(X)\to\rS^\R<\rS$ is (a lift of) the uniformizing representation of $X$, and 
  \item $\rho_\rC:\pi_1(X)\to\rC^\R<\rC$ is any representation into the compact real form of $\rC.$
\end{itemize}   
\end{theorem}
\begin{remark}
  Note that the representations $\rho=\rho_u*\rho_\rC:\pi_1(X)\to\rG$ all factor through the canonical real form of Hodge type $\rG^\R<\rG$ associated to the grading $\bigoplus_{j\in\Z}\fg_j$. Moreover, the $\rG^\R$-centralizer of any such representations is compact since the $\rG^\R$-centralizer of $\rS^\R$ is $\rC^\R$ which is compact. As a result, representations associated to maximal Hodge bundles of type $(\rG_0,\fg_1)$ do not factor through proper parabolic subgroups $\rP^\R<\rG^\R.$
\end{remark}
\begin{proof}
  The proof is immediate from the nonabelian Hodge correspondence and Theorem \ref{thm rigidity jm regular}. Namely, given an even $\fsl_2$-triple $\{f,h,e\}$ with associated grading $\bigoplus_{j\in\Z}\fg_j$, every polystable $(\rG_0,\fg_1)$-Higgs pair $(E,\varphi)$ with Toledo invariant $-\rk_T(\rG_0,\fg_1)(2g-2)$ is isomorphic to $(E_\rT\otimes E_\rC,e)$, where $(E_\rT,e)$ is a uniformizing Higgs bundle of $X$ and $E_\rC$ is a polystable $\rC$-bundle. 
  A solution to the Hitchin equations is given by $h_\rT\otimes h_\rC$, where $h_\rT$ solves the Hitchin equations for the uniformizing Higgs bundle $(E_\rT,e)$ and the Chern connection of $h_\rC$ is flat. 
  The associated flat connection is then 
  \[A_{h_\rT}+e-\tau(e)+A_{h_\rC}.\]

  As a result, the associated representation $\rho:\pi_1(X)\to\rG$ is given by $\rho_u*\rho_\rC$, where $\rho_u:\pi_1(X)\to\rS^\R<\rS$ is the uniformizing representation of $X$ if $\rS\cong\rPSL_2\C$ and a lift of the uniformizing representation if $\rS\cong\rSL_2\C$ and $\rho_\rC:\pi_1(X)\to\rC^\R<\rC$ is a representation into the compact real form of $\rC.$
\end{proof}

 Let $\Sigma$ be a closed topological surface of genus $g\geq 2.$ For $\rS^\R$ isomorphic to $\rPSL_2\R$ or $\rSL_2\R$, injective representations $\rho:\pi_1(\Sigma)\to\rS^\R$ with discrete image are called Fuchsian representations.  The set of conjugacy classes of Fuchsian representations defines an open and closed set in the character variety $\cR(\rS^\R)$, each connected component of which is identified with the Teichm\"uller space of the surface $\Sigma$ by the uniformization theorem.
\begin{corollary}
Suppose $(\rG_0,\fg_1)$ is JM-regular and let $\rho:\pi_1(\Sigma)\to\rG$ be a reductive representation. There exists a Riemann surface structure $X_\rho$ on $\Sigma$ such that the Higgs bundle associated to $\rho$ is a maximal Hodge bundle of type $(\rG_0,\fg_1)$ on $X_\rho$ if and only if 
\[\rho=\rho_{\mathrm{Fuch}}*\rho_{\rC^\R},\] where 
$\rho_{\mathrm{Fuch}}:\pi_1(\Sigma)\to\rS^\R$ is a Fuchsian representations and $\rho_{\rC^\R}:\pi_1(\Sigma)\to\rC^\R$ is a representation into the compact real form of $\rC.$ In particular, for any such representations, the associated Riemann surface $X_\rho$ is unique.
\end{corollary}  

The following Corollary is also immediate from Theorem \ref{thm jm reg rigidity of reps for maximal}.

\begin{corollary}\label{cor JM reg max vhs}
Suppose $(\rG_0,\fg_1)$ is JM-regular. A variation of Hodge structure $(\rho,f_\rho)$ associated to a Hodge bundle of type $(\rG_0,\fg_1)$ is maximal if and only if $f_\rho:\widetilde X\to\rG^\R/\rH_0^\R$ is a totally geodesic embedding which maximizes the holomorphic sectional curvature. In particular, the image of $f_\rho$ is independent of $\rho$ and rigid.
\end{corollary}

\subsection{Some explicit examples}
We list some explicit examples for $(\rG_0,\fg_1)$ JM-regular.
\begin{example}\label{ex principal max Hodge}
   When $\rG_0<\rG$ is a Cartan subgroup, $(\rG_0,\fg_1)$ is a JM-regular phvs associated to a principal $\fsl_2$-triple. In this case, the group $\rC$ is the center of $\rG$. 
   The Higgs bundles $(E_\rC\otimes E_\rT(\rG),e)$ associated to maximal $(\rG_0,\fg_1)$-Higgs pairs are the image of $0$ under the Hitchin section \cite{liegroupsteichmuller}.
 \end{example} 
\begin{example}\label{ex Herm tube case maximal}
  For JM-regular phvs's $(\rG_0,\fg_1)$  associated to gradings $\fg=\fg_{-1}\oplus\fg_0\oplus\fg_1$, the real form $\rG^\R<\rG$ is a Hermitian Lie group of tube type (see Example \ref{ex real forms of examples}). In this case, the above results recover the subset of the results in \cite{BGRmaximalToledo}. In this case, there is a particular stratum of the boundary of the noncompact Riemannian symmetric space $\rG^\R/\rH^\R$ called the Shilov boundary. The group $\rC^\R$ is isomorphic to the stabilizer of a generic point in the Shilov boundary.
\end{example}

\begin{example}\label{ex so(2p,q) jm reg details}
  We now describe the JM-regular $(\rGL_p\C\times\rSO_q\C,M_{p,q})$ with $q\geq p$ from Examples \ref{ex phvs M_p,q}, \ref{ex parabolic type} and \ref{ex real forms of examples} in more detail.  Fix a decomposition $\C^{2p+q}=\C^p\oplus\C^q\oplus\C^p$ with orthogonal structure $Q(x,y,z)=xz^T+x^Tz+yy^T$, then we may write $\fso_{2p+q}\C$ as 
  \[\fso_{2p+q}\C=\{\begin{pmatrix}
    A&W&X\\Y&B&-W^T\\Z&-Y^T&-A^T
  \end{pmatrix}~|~\vcenter{\xymatrix@=0em{A\in M_{p,p},~ B=-B^T\in M_{q,q},~ W\in M_{p,q}~\\ X=-X^T\in M_{p,p},~ Y\in M_{q,p},~ Z=-Z^T\in M_{p,p}}}\}\]
  The even $\fsl_2$-triple is given by 
  \[f=\begin{pmatrix}
    0&0&0\\\smtrx{2\Id_p\\0}&0&0\\0&\smtrx{-2\Id_p&0}&0
  \end{pmatrix}, ~ h=\begin{pmatrix}
    2\Id_p&0&0\\0&0&0\\0&0&-2\Id_p
  \end{pmatrix},~e=\begin{pmatrix}
    0&\smtrx{\Id_p&0}&0\\0&0&\smtrx{-\Id_p\\0}\\0&0&0
  \end{pmatrix},\]
  where $\Id_p$ is the $p\times p$ identity matrix. The $\Z$-grading is $\fg_{-2}\oplus\fg_1\oplus\fg_0\oplus\fg_1\oplus\fg_2$ where the space $\fg_j$ is the $j^{th}$ super block diagonal. 

  The subgroup $\rG_0<\rSO_{2p+q}\C$ is $\rGL_p\C\times\rSO_q\C$ and the $\rG_0$-stabilizer of $e$ is given by 
  \[\rC=\rG_0^e=\{\begin{pmatrix}
    A&0&0\\0&\smtrx{A&0\\0&B}&0\\0&0&(A^{-1})^T
  \end{pmatrix}~|~AA^T=\Id \text{ and } BB^T=\Id\}\cong\rS(\rO_p\C\times\rO_{q-p}\C).\]

  Using the standard representation of $\rSO_{2p+q}\C$, an $\rSO_{2p+q}\C$-Higgs bundle is a triple $(E,Q,\Phi)$, where $E$ is a rank $2p+q$ holomorphic bundle with trivialized determinant bundle, $Q$ is an everywhere nondegenerate symmetric bilinear form on $E$ and $\Phi\in H^(\End(E)\otimes K)$ is $K$-twisted endomorphism of $E$ which is skew symmetric with respect to $Q.$ 
  An $\rSO_{2p+q}\C$ is a Hodge bundle of type $(\rG_0,\fg_1)$ if $E$ splits holomorphically as $E=V\oplus W\oplus V^*$ where $V$ is a rank $p$ isotropic subbundle and $W$ is a rank $q$ orthogonal subbundle, and with respect to this splitting we have 
  \begin{equation}
    \label{eq higgs field form}\Phi=\begin{pmatrix}0&\theta&0\\0&0&-\theta^T\\0&0&0
  \end{pmatrix}:V\oplus W\oplus V^*\to(V\oplus W\oplus V^*)\otimes K
  \end{equation}
    
  Such a Higgs bundle is a maximal $(\rG_0,\fg_1)$-Higgs pair if and only if $W$ decomposes holomorphically and orthogonally as $W=U_p\oplus U_{q-p}$ and $V=U_p\otimes K^{-1}$ and 
  \[\theta=\begin{pmatrix}
    \Id_{U_p}&0
  \end{pmatrix}:W=U_p\oplus U_{q-p} \to U_p=V\otimes K\]
  
  In this example, the group $\rC=\rS(\rO_p\C\times \rO_{q-p}\C)$ and the space of such Higgs bundles is parameterized by the polystable $\rS(\rO_p\C\times \rO_{q-p}\C)$ bundle $U_p\oplus U_{q-p}.$ These Higgs bundles reduce to $\rSO(2p,q)$-Higgs bundles and the resulting representations $\rho:\pi_1(X)\to\rSO(2p,q)$ are conjugate to $\rho_u*\rho_{\rC^\R}$ where $\rho_u:\pi_1(X)\to\rPSL_2\R$ is the uniformizing representation of $X$ and $\rho_{\rC^\R}:\pi_1(X)\to\rS(\rO_p\times\rO_{q-p})$ is any representation. When $p=1,$ this is agrees with Example \ref{ex Herm tube case maximal} for the Hermitian Lie group $\rSO(2,q).$
  \end{example}
  \begin{remark}
      For the JM-regular phvs $(\rS(\rGL_p\C\times\rGL_q\C\times \rGL_p\C),M_{p,q}\oplus M_{q,p})$ with $q\geq p$ from Examples \ref{ex phvs M_p,q}, \ref{ex parabolic type} and \ref{ex real forms of examples}, the moduli space $\rSL_{2p+q}\C$-Higgs bundles which are maximal Hodge bundles of type $(\rG_0,\fg_1)$ can be described similarly. Namely, they are given by pairs $(E,\Phi)$ where $E$ is a holomorphic rank $2p+q$ vector bundle with trivialized determinant and $\Phi$ is a traceless $K$-twisted endomorphism of $E$ such that $E=(U_p\otimes K^{-1})\oplus U_p\oplus U_{q-p}\oplus (U_p\otimes K)$, $\det(U_p)^3=\det(U_{q-p})$ and 
      \[\Phi=\begin{pmatrix}0&\Id_p&0&0\\0&0&0&\Id_p\\0&0&0&0\\0&0&0&0
      \end{pmatrix}:E\to E\otimes K.\]
  \end{remark}

\begin{example}
There is a class of nilpotent elements called distinguished nilpotents which play an important role in the Bala--Carter theory classification of nilpotent orbits in $\fg$. A nilpotent $e\in\fg$ is distinguished if and only if the centralizer of an associated $\fsl_2$-triple $\rC<\rG$ is discrete. 
Such $\fsl_2$-triples are necessarily even, and hence the resulting moduli space of maximal JM-regular $(\rG_0,\fg_1)$-Higgs pairs is discrete. In particular, the associated $\C^*$-fixed points in the $\rG$-Higgs bundle moduli space are isolated. 

The principal $\fsl_2$-triples are always distinguished and give rise to the isolated fixed points from Example \ref{ex principal max Hodge}. In type $\rA$, this is the only distinguished nilpotent. However, in all other types there are many distinguished nilpotents. For example, for $\fg=\fso_n\C$ a nilpotent is distinguished if and only if it is associated a partition $n=1\cdot n_1+\cdots +1\cdot n_k,$ where each $n_j$ is odd. 
For $\fg=\fsp_{2n}\C$ a nilpotent is distinguished if and only if it is associated to a partition $2n=1\cdot n_1+\cdots +1\cdot n_k$, with each $n_j$ even. Moreover, for the Lie group $\rG_2$ there are two even $\fsl_2$-triples and both are distinguished. See \cite[\S6 \&\S8]{CollMcGovNilpotents} for more details. This means that the moduli space of $\rG$-Higgs bundles has isolated fixed points which do not arise from the Hitchin section when $\fg$ is not of type $\rA.$
\end{example}

\begin{example}
 In \cite{MagicalBCGGO}, a class of even nilpotents is identified called magical nilpotents. For such even $\fsl_2$-triples, there is a real form $\rG^\R<\rG$ such that the associated maximal $(\rG_0,\fg_1)$-Higgs pairs define $\rG^\R$-Higgs bundles which are local minima of the energy function on the moduli space $\rG^\R$-Higgs bundles. For such Hodge bundles, the global Slodowy slice map from \cite{ColSandGlobalSlodowy} descends to moduli spaces and describes connected components of the moduli space $\cM(\rG^\R)$ with many interesting properties. In particular, the associated components of the character variety $\cR(\rG^\R)$ define the higher Teichm\"uller components conjectured by Guichard--Labourie--Wienhard \cite{PosRepsGLW,PosRepsGWPROCEEDINGS}. There are four families of magical nilpotents two of which appear in Examples \ref{ex principal max Hodge} and \ref{ex Herm tube case maximal}.
\end{example}

\subsection{The non JM-regular case}\label{sec non JM reg max Hodge}
We now describe the moduli space of maximal $(\rG_0,\fg_1)$-Higgs pairs when $(\rG_0,\fg_1)$ is not JM-regular. 
Assume $(\rG_0,\fg_1)$ is not a JM-regular phvs. Let $\Omega\subset\fg_1$ be the open $\rG_0$-orbit. For $e\in\Omega$ we complete it to an $\fsl_2$-triple $\{f,h,e\}$ with $h\in\fg_0$ and set $s=\zeta-\frac{h}{2}$. Recall $s\neq0$ exactly when $(\rG_0,\fg_1)$ is not JM-regular. 
A maximal JM-regular phvss $(\hat\rG_0,\hat\fg_1)$ for $(\rG_0,\fg_1)$ is given by letting $\hat\rG_0<\rG_0$ be the $\rG_0$-centralizer of $s$ and $\hat\fg_1=\{x\in\fg_1~|~[s,x]=0\}$. Recall also that $s$ determines a parabolic subgroup $\rP_s<\rG_0$ and $\hat\rG_0<\rP_s$ is the Levi subgroup.

\begin{proposition}\label{prop no stable reduce to max JMreg}
  Assume $(\rG_0,\fg_1)$ is not JM-regular and let $(\hat\rG_0,\hat\fg_1)$ be a maximal JM-regular prehomogeneous vector subspace. Then, there are no stable $(\rG_0,\fg_1)$-Higgs pairs which are maximal. Moreover, every maximal polystable $(\rG_0,\fg_1)$-Higgs pair with Toledo invariant reduces to a polystable maximal $(\hat\rG_0,\hat\fg_1)$-Higgs pair.
 \end{proposition}
\begin{proof}
  Let $(E,\varphi)$ be a polystable maximal $(\rG_0,\fg_1)$-Higgs pair. Since $(E,\varphi)$ is maximal, $\varphi(x)$ is in the open orbit $\Omega$ for generic $x\in X.$ Let $e\in\Omega$, $\{f,h,e\}$ be an $\fsl_2$-triple with $h\in\fg_0$ and $(\hat\rG_0,\hat\fg_1)$ be the associated maximal JM-regular phvss of $(\rG_0,\fg_1).$ Let $s=\zeta-\frac{h}{2}$, and $\rP_s<\rG_0$ be the associated parabolic subgroup.
  As in the proof of Theorem \ref{thm: Arakelov--Milnor ineq}, the Higgs field defines a reduction of structure group $\sigma$ of $E$ to a $\rP_s$-bundle $E_\sigma$ such that $\varphi\in H^0(E_\sigma(\fg_{1,s})\otimes K).$ 
  Since $(E,\varphi)$ is maximal, we have $\deg E(\sigma,s)=0$. Hence $(E,\varphi)$ is not stable and there is a holomorphic reduction $\hat\sigma$ to the Levi subgroup $\hat\rG_0<\rP_s$ with $\varphi\in H^0(E_{\hat\sigma}(\hat\fg_1)\otimes K.$ Thus, $(E,\varphi)$  reduces to a polystable maximal $(\hat\rG_0,\hat\fg_1)$-Higgs pair.
\end{proof}

Let $\hat\rC<\hat\rG_0$ be the $\hat\rG_0$-stabilizer of the $\fsl_2$-triple $\{f,h,e\}$. 
Note that $\{f,h,e,s\}\subset\fg$ generate a subalgebra isomorphic to $\fgl_2\C$, and $\hat\rC$ is the $\rG$-centralizer of the $\{f,h,e,s\}.$ 
As in the JM-regular case, extending the structure group of the uniformizing Higgs bundle $(E_\rT,e)$ to $\rG_0$ defines a polystable maximal $(\rG_0,\fg_1)$-Higgs pair
$(E_\rT(\rG_0),e)$.
Moreover, if $E_{\hat\rC}$ is a polystable $\hat\rC$-bundle, then $(E_T\otimes E_{\hat\rC}(\rG_0),e)$ is a polystable maximal $(\rG_0,\fg_1)$-Higgs pair. 
Thus we have a map 
\begin{equation}
  \label{eq: Psi hat e map}
  \hat\Psi_e:\xymatrix@R=0em{\cN(\hat\rC)\ar[r]&\cM^{\max}(\rG_0,\fg_1)\\E_{\hat\rC}\ar@{|->}[r]&(E_\rT\otimes E_{\hat\rC}(\rG_0),e)}
\end{equation}
from the moduli space of polystable $\hat\rC$-bundles on $X$ to the moduli space of maximal $(\rG_0,\fg_1)$-Higgs pairs.

\begin{theorem}\label{thm param of nonJM reg max}
If $(\rG_0,\fg_1)$ is not JM-regular, then map $\hat\Psi_e:\cN(\hat\rC)\to\cM^{\max}(\rG_0,\fg_1)$ from \eqref{eq: Psi hat e map} defines an isomorphism between the moduli space of polystable $\hat\rC$-bundles and the moduli space of polystable maximal $(\rG_0,\fg_1)$-Higgs pairs. 
\end{theorem}
\begin{proof}
Surjectivity of the map $\hat\Psi_e$ follows from Proposition \ref{prop no stable reduce to max JMreg} and the proof of surjectivity in JM-regular case. For injectivity, note that two polystable $\hat\rC$-bundles $E_{\hat\rC}$ and $E'_{\hat\rC}$ define the same point in $\cM^{\max}(\rG_0,\fg_1)$ if there is a holomorphic isomorphism of $\rG_0$-bundles
\[\Phi:E_{\hat\rC}\otimes E_\rT(\rG_0)\to E_{\hat\rC}'\otimes E_\rT(\rG_0)\]
such that $\Phi^*e=e.$ Unlike the JM-regular case, the $\rG_0$-stabilizer of $e$ is not $\hat\rC$. However,  $E_{\hat\rC}\otimes E_\rT(\rG_0)\cong E_{\hat\rC}'\otimes E_\rT(\rG_0)$ implies that $\Phi$ acts trivially on $E_\rT$. 
In particular, $\Phi$ is valued in the $\rG_0$-centralizer of $\{h,e\}$. By the uniqueness part of the Jacobson--Morozov theorem, this is the $\rG_0$-centralizer of the $\fsl_2$-triple $\{f,h,e\}$ which is $\hat\rC.$ Hence $\Phi$ induces a holomorphic isomorphism of the $\hat\rC$-bundles.
\end{proof}

The rigidity results for variations of Hodge structure in the JM-regular cases have analogues in the non JM-regular setting. 
Namely, Consider the Cartan involution $\theta:\fg\to\fg$ which is $(-1)^j\Id$ on $\fg_j$ and choose the compact real form $\tau:\fg\to\fg$ such that $\tau(\fg_j=\fg_{-j}$ and $\tau(e)=-f.$ 
Let $\rG^\R<\rG$, $\rS^\R<\rS$, $\hat\rC^\R<\hat\rC$ the associated real forms. The proof of the following theorem is immediate.

\begin{theorem}\label{rigidity non JM reg case}
  Suppose $(\rG_0,\fg_1)$ is not JM-regular, and $\rho:\pi_1(X)\to\rG$ is a reductive representation. The Higgs bundle associated to $\rho$ is a maximal Hodge bundle of type $(\rG_0,\fg_1)$ if and only if, up to conjugation, $\rho=\rho_u*\rho_{\hat\rC}:\pi_1(X)\to\rG$, where
  \begin{itemize}
    \item $\rho_u:\pi_1(X)\to\rS^\R<\rS$ is (a lift of) the uniformizing representation of $X$, and
    \item $\rho_{\hat\rC}:\pi_1(X)\to\hat\rC^\R<\hat\rC$ is any representation into the compact real form of $\hat\rC.$
  \end{itemize}
\end{theorem}
\begin{remark}
 As in the JM-regular case, such representations factor through the real form $\rG^\R<\rG.$ However, unlike the JM-regular setting, the $\rG^\R$-centralizer of these representations is not necessarily compact since $\hat\rC$ is not the full $\rG$-centralizer of the $\fsl_2$-triple $\{f,h,e\}.$
\end{remark}
Note that Corollary \ref{cor JM reg max vhs} now holds without the JM-regular assumption.

\begin{corollary}\label{cor JM reg max vhs}
Consider a phvs $(\rG_0,\fg_1)$. A variation of Hodge structure $(\rho,f_\rho)$ associated to a maximal Hodge bundle of type $(\rG_0,\fg_1)$ if and only if $f_\rho:\widetilde X\to\rG^\R/\rH_0^\R$ is a totally geodesic embedding which maximizes the holomorphic sectional curvature. 
If $(\rG_0,\fg_1)$ is not JM-regular then $f_\rho$ factors through a maximal JM-regular subdomain
\[f_\rho:\widetilde X\to\hat\rG^\R/\hat\rH_0^\R\hookrightarrow \rG^\R/\rH_0^\R.\]
\end{corollary}
\begin{example}
 Recall from example \ref{ex phvs M_p,q} and \ref{ex max jm reg} that the phvs $(\rGL_p\C\times \rSO_q\C,M_{p,q})$ with $q<p$ is not JM-regular and a maximal JM-regular phvss is isomorphic to $(\rGL_{q}\C\times \rGL_{p-q}\C\times \rSO_q\C,M_{q,q}).$ 
Following example \ref{ex so(2p,q) jm reg details}, an $\rSO_{2p+q}\C$-Higgs bundle $(E,Q,\Phi)$ is a Hodge bundle of type $(\rG_0,\fg_1)$ if and only if $E$ splits holomorphically as $V\oplus W\oplus V^*$ where $V$ is a rank $p$ isotropic subbundle and the Higgs field is determined by a holomorphic map $\theta:W\to V\otimes K$  as in \eqref{eq higgs field form}.
Such a Higgs bundle is a maximal $(\rG_0,\fg_1)$-Higgs pair if and only if $V$ decomposes holomorphically as $V=WK^{-1}\oplus V_{q-p}$ where $\rk(V_q-p)$ is a rank $q-p$ degree zero polystable vector bundle and
\[\theta=\begin{pmatrix}
  \Id_{W}\\0
\end{pmatrix}W\to (WK^{-1}\oplus V_{q-p})\otimes K= V\otimes K.\]
The resulting representations $\rho:\pi_1(X)\to\rSO_{2p+q}\C$ factors through $\rSO(2q,q)\times \rU_{p-q}<\rSO(2p,q)<\rSO_{2p+q}\C.$
\end{example}

\bibliographystyle{plain}
\bibliography{am_inequalities.bib}

\begin{thebibliography}{10}

\bibitem{BGRmaximalToledo}
Olivier Biquard, Oscar García-Prada, and Roberto Rubio.
\newblock Higgs bundles, the {T}oledo invariant and the {C}ayley
  correspondence.
\newblock {\em Journal of Topology}, 10(3):795--826, 2017.

\bibitem{MagicalBCGGO}
Steven Bradlow, Brian Collier, Oscar Garc{\'i}a-Prada, Peter~B. Gothen, and
  Andr{\'e} Oliveira.
\newblock A general {C}ayley correspondence for higher {T}eichm{\"u}ller
  spaces.
\newblock {\em In preparation}, 2020.

\bibitem{UpqHiggs}
Steven~B. Bradlow, Oscar Garc{\'{\i}}a-Prada, and Peter~B. Gothen.
\newblock Surface group representations and {${\rm U}(p,q)$}-{H}iggs bundles.
\newblock {\em J. Differential Geom.}, 64(1):111--170, 2003.

\bibitem{RelKHpairscorresp}
Steven~B. Bradlow, Oscar Garc{\'{\i}}a-Prada, and Ignasi Mundet~i Riera.
\newblock Relative {H}itchin-{K}obayashi correspondences for principal pairs.
\newblock {\em Q. J. Math.}, 54(2):171--208, 2003.

\bibitem{Burger-Iozzi-Wienhard-CR}
Marc Burger, Alessandra Iozzi, and Anna Wienhard.
\newblock Surface group representations with maximal {T}oledo invariant.
\newblock {\em C. R. Math. Acad. Sci. Paris}, 336(5):387--390, 2003.

\bibitem{BIWmaximalToledoAnnals}
Marc Burger, Alessandra Iozzi, and Anna Wienhard.
\newblock Surface group representations with maximal {T}oledo invariant.
\newblock {\em Ann. of Math. (2)}, 172(1):517--566, 2010.

\bibitem{PeriodDomainsCMP}
James Carlson, Stefan M\"{u}ller-Stach, and Chris Peters.
\newblock {\em Period mappings and period domains}, volume 168 of {\em
  Cambridge Studies in Advanced Mathematics}.
\newblock Cambridge University Press, Cambridge, 2017.
\newblock Second edition.

\bibitem{ColSandGlobalSlodowy}
Brian Collier and Andrew Sanders.
\newblock ({G},{P})-opers and global {S}lodowy slices.
\newblock {\em Advances in Mathematics}, 377:107490, 2021.

\bibitem{CollMcGovNilpotents}
David~H. Collingwood and William~M. McGovern.
\newblock {\em Nilpotent orbits in semisimple {L}ie algebras}.
\newblock Van Nostrand Reinhold Mathematics Series. Van Nostrand Reinhold Co.,
  New York, 1993.

\bibitem{canonicalmetrics}
Kevin Corlette.
\newblock Flat {$G$}-bundles with canonical metrics.
\newblock {\em J. Differential Geom.}, 28(3):361--382, 1988.

\bibitem{MilnorWoodIneqDomicToledo}
Antun Domic and Domingo Toledo.
\newblock The {G}romov norm of the {K}aehler class of symmetric domains.
\newblock {\em Math. Ann.}, 276(3):425--432, 1987.

\bibitem{harmoicmetric}
Simon Donaldson.
\newblock Twisted harmonic maps and the self-duality equations.
\newblock {\em Proc. London Math. Soc. (3)}, 55(1):127--131, 1987.

\bibitem{HiggsPairsSTABILITY}
Oscar Garc{\'{\i}}a-Prada, Peter Gothen, and Ignasi Mundet~i Riera.
\newblock {The {H}itchin-{K}obayashi correspondence, {H}iggs pairs and surface
  group representations}.
\newblock {\em ArXiv e-prints 0909.4487}, September 2009.

\bibitem{GPHygenus}
Oscar Garc\'{\i}a-Prada and Jochen Heinloth.
\newblock The {$y$}-genus of the moduli space of {${\rm PGL}_n$}-{H}iggs
  bundles on a curve (for degree coprime to {$n$}).
\newblock {\em Duke Math. J.}, 162(14):2731--2749, 2013.

\bibitem{GPHSMotivesHiggs}
Oscar Garc\'{\i}a-Prada, Jochen Heinloth, and Alexander Schmitt.
\newblock On the motives of moduli of chains and {H}iggs bundles.
\newblock {\em J. Eur. Math. Soc. (JEMS)}, 16(12):2617--2668, 2014.

\bibitem{GothenBettirk3}
Peter~B. Gothen.
\newblock The {B}etti numbers of the moduli space of stable rank {$3$} {H}iggs
  bundles on a {R}iemann surface.
\newblock {\em Internat. J. Math.}, 5(6):861--875, 1994.

\bibitem{PosRepsGLW}
Olivier Guichard, Fran{\c{c}}ois Labourie, and Anna Wienhard.
\newblock Positive representations.
\newblock {\em In preparation}, 2016.

\bibitem{PosRepsGWPROCEEDINGS}
Olivier Guichard and Anna Wienhard.
\newblock Positivity and higher {T}eichm{\"u}ller theory.
\newblock {\em Proceedings of the 7th {E}uropean {C}ongress of {M}athematics},
  2016.

\bibitem{selfduality}
Nigel Hitchin.
\newblock The self-duality equations on a {R}iemann surface.
\newblock {\em Proc. London Math. Soc. (3)}, 55(1):59--126, 1987.

\bibitem{liegroupsteichmuller}
Nigel Hitchin.
\newblock Lie groups and {T}eichm\"uller space.
\newblock {\em Topology}, 31(3):449--473, 1992.

\bibitem{Jost-Zuo}
J\"{u}rgen Jost and Kang Zuo.
\newblock Arakelov type inequalities for {H}odge bundles over algebraic
  varieties. {I}. {H}odge bundles over algebraic curves.
\newblock {\em J. Algebraic Geom.}, 11(3):535--546, 2002.

\bibitem{KimuraIntroPHVS}
Tatsuo Kimura.
\newblock {\em Introduction to prehomogeneous vector spaces}, volume 215 of
  {\em Translations of Mathematical Monographs}.
\newblock American Mathematical Society, Providence, RI, 2003.

\bibitem{knappbeyondintro}
Anthony~W. Knapp.
\newblock {\em Lie groups beyond an introduction}, volume 140 of {\em Progress
  in Mathematics}.
\newblock Birkh\"auser Boston Inc., Boston, MA, second edition, 2002.

\bibitem{Qiongling_Li}
Qiongling {Li}.
\newblock {Nilpotent {H}iggs bundles and the {H}odge metric on the
  {C}alabi-{Y}au moduli}.
\newblock {\em arXiv e-prints}, page arXiv:2005.13939, May 2020.

\bibitem{ManievelPrehom}
L.~Manivel.
\newblock Prehomogeneous spaces and projective geometry.
\newblock {\em Rend. Semin. Mat. Univ. Politec. Torino}, 71(1):35--118, 2013.

\bibitem{MilnorMWinequality}
John Milnor.
\newblock On the existence of a connection with curvature zero.
\newblock {\em Comment. Math. Helv.}, 32:215--223, 1958.

\bibitem{MollerViehwegZuo}
Martin M\"{o}ller, Eckart Viehweg, and Kang Zuo.
\newblock Stability of {H}odge bundles and a numerical characterization of
  {S}himura varieties.
\newblock {\em J. Differential Geom.}, 92(1):71--151, 2012.

\bibitem{MortajineRedbook}
A.~Mortajine.
\newblock {\em Classification des espaces pr\'{e}homog\`enes de type
  parabolique r\'{e}guliers et de leurs invariants relatifs}, volume~40 of {\em
  Travaux en Cours [Works in Progress]}.
\newblock Hermann, Paris, 1991.

\bibitem{NarasimhanSeshadri}
M.~S. Narasimhan and C.~S. Seshadri.
\newblock Stable and unitary vector bundles on a compact {R}iemann surface.
\newblock {\em Ann. of Math. (2)}, 82:540--567, 1965.

\bibitem{PetersRigidityVHSArakelov}
C.~A.~M. Peters.
\newblock Rigidity for variations of {H}odge structure and {A}rakelov-type
  finiteness theorems.
\newblock {\em Compositio Math.}, 75(1):113--126, 1990.

\bibitem{ramanathan_1975}
Annamalai Ramanathan.
\newblock Stable principal bundles on a compact {R}iemann surface.
\newblock {\em Mathematische Annalen}, 213(2):129--152, 1975.

\bibitem{RubenthalerNonJMRegPHVS}
Hubert Rubenthaler.
\newblock Decomposition of reductive regular prehomogeneous vector spaces.
\newblock {\em Ann. Inst. Fourier (Grenoble)}, 61(5):2183--2218 (2012), 2011.

\bibitem{SatoKimuraPHVS}
M.~Sato and T.~Kimura.
\newblock A classification of irreducible prehomogeneous vector spaces and
  their relative invariants.
\newblock {\em Nagoya Math. J.}, 65:1--155, 1977.

\bibitem{schmittGITbook}
Alexander H.~W. Schmitt.
\newblock {\em Geometric invariant theory and decorated principal bundles}.
\newblock Zurich Lectures in Advanced Mathematics. European Mathematical
  Society (EMS), Z\"{u}rich, 2008.

\bibitem{SimpsonVHS}
Carlos~T. Simpson.
\newblock Constructing variations of {H}odge structure using {Y}ang-{M}ills
  theory and applications to uniformization.
\newblock {\em J. Amer. Math. Soc.}, 1(4):867--918, 1988.

\bibitem{localsystems}
Carlos~T. Simpson.
\newblock Higgs bundles and local systems.
\newblock {\em Inst. Hautes \'Etudes Sci. Publ. Math.}, 75:5--95, 1992.

\bibitem{SimpsonModuli1}
Carlos~T. Simpson.
\newblock Moduli of representations of the fundamental group of a smooth
  projective variety. {I}.
\newblock {\em Inst. Hautes \'Etudes Sci. Publ. Math.}, 79:47--129, 1994.

\bibitem{centralizerSpringerSteinberg}
T.~A. Springer and R.~Steinberg.
\newblock Conjugacy classes.
\newblock In {\em Seminar on {A}lgebraic {G}roups and {R}elated {F}inite
  {G}roups ({T}he {I}nstitute for {A}dvanced {S}tudy, {P}rinceton, {N}.{J}.,
  1968/69)}, Lecture Notes in Mathematics, Vol. 131, pages 167--266. Springer,
  Berlin, 1970.

\bibitem{ToledoSU(1n)Rigid}
Domingo Toledo.
\newblock Representations of surface groups in complex hyperbolic space.
\newblock {\em J. Differential Geom.}, 29(1):125--133, 1989.

\bibitem{viehweg}
Eckart Viehweg.
\newblock Arakelov inequalities.
\newblock In {\em Surveys in differential geometry. {V}ol. {XIII}. {G}eometry,
  analysis, and algebraic geometry: forty years of the {J}ournal of
  {D}ifferential {G}eometry}, volume~13 of {\em Surv. Differ. Geom.}, pages
  245--275. Int. Press, Somerville, MA, 2009.

\bibitem{ViehwegZuo}
Eckart Viehweg and Kang Zuo.
\newblock Families over curves with a strictly maximal {H}iggs field.
\newblock {\em Asian J. Math.}, 7(4):575--598, 2003.

\bibitem{LieIII}
{\`E}.~B. Vinberg, editor.
\newblock {\em Lie groups and {L}ie algebras, {III}}, volume~41 of {\em
  Encyclopaedia of Mathematical Sciences}.
\newblock Springer-Verlag, Berlin, 1994.
\newblock Structure of Lie groups and Lie algebras, A translation of {{\i}t
  Current problems in mathematics. Fundamental directions. Vol. 41} (Russian),
  Akad. Nauk SSSR, Vsesoyuz. Inst. Nauchn. i Tekhn. Inform., Moscow, 1990 [
  MR1056485 (91b:22001)], Translation by V. Minachin [V. V. Minakhin],
  Translation edited by A. L. Onishchik and {\`E}. B. Vinberg.

\end{thebibliography}
\end{document}